\documentclass[a4paper,twoside]{article}
\usepackage{a4}
\usepackage{amssymb}
\usepackage{amsmath}
\usepackage{upref}
\usepackage[active]{srcltx}
\allowdisplaybreaks[2] 
\usepackage[colorlinks,citecolor=blue,linkcolor=blue]{hyperref}
\usepackage[dvipsnames]{color}
\usepackage{pgf,tikz}
\usepackage{subfig}
\usepackage{mathrsfs}
\usetikzlibrary{arrows}

\definecolor{blau}{rgb}{0.1,0.0,0.9}
\definecolor{gruen}{cmyk}{1.0,0.2,0.7,0.07}
\definecolor{mag}{cmyk}{0.0,0.9,0.3,0.0}

\newcount\minutes \newcount\hours
\hours=\time
\divide\hours 60
\minutes=\hours
\multiply\minutes -60
\advance\minutes \time
\newcommand{\klockan}{\the\hours:{\ifnum\minutes<10 0\fi}\the\minutes}
\newcommand{\tid}{\today\ \klockan}
\newcommand{\prtid}{\smash{\raise 10mm \hbox{\LaTeX ed \tid}}}
\makeatletter
\def\sectionmark#1{} 
\def\subsectionmark#1{}
\newcommand{\sectnr}{\ifnum \c@secnumdepth >\z@
                 \thesection.\hskip 1em\relax \fi}
\def\@evenhead{\footnotesize\rm\thepage\hfil\leftmark\hfil\llap{\prtid}}
\def\@oddhead{\footnotesize\rm\rlap{\prtid}\hfil\rightmark\hfil\thepage}
\def\tableofcontents{\section*{Contents} 
 \@starttoc{toc}}
\makeatother
\makeatletter
\def\@biblabel#1{#1.}
\makeatother
\makeatletter
\let\Thebibliography=\thebibliography
\renewcommand{\thebibliography}[1]{\def\@mkboth##1##2{}\Thebibliography{#1}
\addcontentsline{toc}{section}{References}
\frenchspacing 
\setlength{\@topsep}{0pt}
\setlength{\itemsep}{0pt}%
\setlength{\parskip}{0pt plus 2pt}%
}
\makeatother
\makeatletter
\def\mdots@{\mathinner.\nonscript\!.%
 \ifx\next,.\else\ifx\next;.\else\ifx\next..\else
 \nonscript\!\mathinner.\fi\fi\fi}
\let\ldots\mdots@
\let\cdots\mdots@
\let\dotso\mdots@
\let\dotsb\mdots@
\let\dotsm\mdots@
\let\dotsc\mdots@
\def\vdots{\vbox{\baselineskip2.8\p@ \lineskiplimit\z@
    \kern6\p@\hbox{.}\hbox{.}\hbox{.}\kern3\p@}}
\def\ddots{\mathinner{\mkern1mu\raise8.6\p@\vbox{\kern7\p@\hbox{.}}%
    \raise5.8\p@\hbox{.}\raise3\p@\hbox{.}\mkern1mu}}
\makeatother
\makeatletter
\let\Enumerate=\enumerate
\renewcommand{\enumerate}{\Enumerate%
\setlength{\@topsep}{0pt}
\setlength{\itemsep}{0pt}%
\setlength{\parskip}{0pt plus 1pt}%
\renewcommand{\theenumi}{\textup{(\alph{enumi})}}%
\renewcommand{\labelenumi}{\theenumi}%
}
\let\endEnumerate=\endenumerate
\renewcommand{\endenumerate}{\endEnumerate\unskip}
\makeatother
\makeatletter
\def\@seccntformat#1{\csname the#1\endcsname.\quad}
\makeatother
\makeatletter
\long\def\@makecaption#1#2{%
  \vskip\abovecaptionskip
  \sbox\@tempboxa{ #1. #2}%
  \ifdim \wd\@tempboxa >\hsize
    #1. #2\par
  \else
    \global \@minipagefalse
    \hb@xt@\hsize{\hfil\box\@tempboxa\hfil}%
  \fi
  \vskip\belowcaptionskip}
\makeatother
\newcommand{\authortitle}[2]{\author{#1}\title{#2}\markboth{#1}{#2}}
\newcommand{\art}[6]{{\sc #1, \rm #2, \it #3 \bf #4 \rm (#5), \mbox{#6}.}}

\newcommand{\book}[3]{{\sc #1, \it #2, \rm #3.}}
\newcommand{\AND}{{\rm and }}
\RequirePackage{amsthm}
\newtheoremstyle{descriptive}%
  {\topsep}   
  {\topsep}   
  {\rmfamily} 
  {}          
  {\bfseries} 
  {.}         
  { }         
  {}          
\newtheoremstyle{propositional}%
  {\topsep}   
  {\topsep}   
  {\itshape}  
  {}          
  {\bfseries} 
  {.}         
  { }         
  {}          
\newtheoremstyle{remarkstyle}%
  {\topsep}   
  {\topsep}   
  {\rmfamily}  
  {}          
  {\itshape} 
  {.}         
  { }         
  {}          
\theoremstyle{propositional}
\newtheorem{thm}{Theorem}[section]
\newtheorem{prop}[thm]{Proposition}
\newtheorem{lem}[thm]{Lemma}
\newtheorem{cor}[thm]{Corollary}
\theoremstyle{descriptive}
\newtheorem{deff}[thm]{Definition}
\newtheorem{example}[thm]{Example}
\newtheorem{assumption}[thm]{Assumption}

\newtheorem{remk}[thm]{Remark}

\makeatletter
\renewenvironment{proof}[1][\proofname]{\par
  \pushQED{\qed}%
  \normalfont
  \trivlist
  \item[\hskip\labelsep
        \itshape
    #1\@addpunct{.}]\ignorespaces
}{%
  \popQED\endtrivlist\@endpefalse
}
\makeatother
\makeatletter
\newcommand{\AppendicesFromNowOn}
  {\renewcommand \thesection {\@Alph\c@section}  
\newcommand\Appendix{\@startsection {section}{1}{\z@}%
                                   {-3.5ex \@plus -1ex \@minus -.2ex}%
                                   {2.3ex \@plus.2ex}%
                                   {\normalfont\Large\bfseries Appendix }}
   \setcounter{section}{0}
  }
\makeatother
\newdimen\extrawidth
\def\iintlim#1#2{\setbox0\hbox{$\scriptstyle#1$}%
        \setbox1\hbox{$\scriptstyle#2$}%
        \extrawidth=\wd1 \advance\extrawidth-\wd0
        \ifdim\extrawidth<0pt \extrawidth=0pt\fi%
        \int_{#1\kern\extrawidth \kern .5em}^{#2\kern -\wd1} \kern -.5em%
}
\renewcommand{\subsetneq}{\varsubsetneq}
\renewcommand{\emptyset}{\varnothing}
\def\vint{\mathop{\mathchoice%
          {\setbox0\hbox{$\displaystyle\intop$}\kern 0.22\wd0%
           \vcenter{\hrule width 0.6\wd0}\kern -0.82\wd0}%
          {\setbox0\hbox{$\textstyle\intop$}\kern 0.2\wd0%
           \vcenter{\hrule width 0.6\wd0}\kern -0.8\wd0}%
          {\setbox0\hbox{$\scriptstyle\intop$}\kern 0.2\wd0%
           \vcenter{\hrule width 0.6\wd0}\kern -0.8\wd0}%
          {\setbox0\hbox{$\scriptscriptstyle\intop$}\kern 0.2\wd0%
           \vcenter{\hrule width 0.6\wd0}\kern -0.8\wd0}}%
          \mathopen{}\int}
\def\vintslides{\mathop{\mathchoice%
          {\setbox0\hbox{$\displaystyle\intop$}\kern 0.22\wd0%
           \vcenter{\hrule height 0.04em width 0.6\wd0}\kern -0.82\wd0}%
          {\setbox0\hbox{$\textstyle\intop$}\kern 0.2\wd0%
           \vcenter{\hrule height 0.04em width 0.6\wd0}\kern -0.8\wd0}%
          {\setbox0\hbox{$\scriptstyle\intop$}\kern 0.2\wd0%
           \vcenter{\hrule height 0.04em width 0.6\wd0}\kern -0.8\wd0}%
          {\setbox0\hbox{$\scriptscriptstyle\intop$}\kern 0.2\wd0%
           \vcenter{\hrule height 0.04em width 0.6\wd0}\kern -0.8\wd0}}%
          \mathopen{}\int}
\newcommand{\Cp}{{C_p}}
\newcommand{\bCp}{{\protect\itoverline{C}_p}}
\DeclareMathOperator{\diam}{diam}

\DeclareMathOperator{\dist}{dist}

\DeclareMathOperator{\Lip}{Lip}

\DeclareMathOperator{\spt}{supp}
\newcommand{\supp}{\spt}

\newcommand{\bdry}{\partial}
\newcommand{\bdy}{\bdry}

\newcommand{\pip}{\varphi}
{\catcode`p =12 \catcode`t =12 \gdef\eeaa#1pt{#1}}      
\def\accentadjtext#1{\setbox0\hbox{$#1$}\kern   
                \expandafter\eeaa\the\fontdimen1\textfont1 \ht0 }
\def\accentadjscript#1{\setbox0\hbox{$#1$}\kern 
                \expandafter\eeaa\the\fontdimen1\scriptfont1 \ht0 }
\def\accentadjscriptscript#1{\setbox0\hbox{$#1$}\kern   
                \expandafter\eeaa\the\fontdimen1\scriptscriptfont1 \ht0 }
\def\accentadjtextback#1{\setbox0\hbox{$#1$}\kern       
                -\expandafter\eeaa\the\fontdimen1\textfont1 \ht0 }
\def\accentadjscriptback#1{\setbox0\hbox{$#1$}\kern     
                -\expandafter\eeaa\the\fontdimen1\scriptfont1 \ht0 }
\def\accentadjscriptscriptback#1{\setbox0\hbox{$#1$}\kern 
                -\expandafter\eeaa\the\fontdimen1\scriptscriptfont1 \ht0 }
\def\itoverline#1{{\mathsurround0pt\mathchoice
        {\rlap{$\accentadjtext{\displaystyle #1}
                \accentadjtext{\vrule height1.593pt}
                \overline{\phantom{\displaystyle #1}
                \accentadjtextback{\displaystyle #1}}$}{#1}}
        {\rlap{$\accentadjtext{\textstyle #1}
                \accentadjtext{\vrule height1.593pt}
                \overline{\phantom{\textstyle #1}
                \accentadjtextback{\textstyle #1}}$}{#1}}
        {\rlap{$\accentadjscript{\scriptstyle #1}
                \accentadjscript{\vrule height1.593pt}
                \overline{\phantom{\scriptstyle #1}
                \accentadjscriptback{\scriptstyle #1}}$}{#1}}
        {\rlap{$\accentadjscriptscript{\scriptscriptstyle #1}
                \accentadjscriptscript{\vrule height1.593pt}
                \overline{\phantom{\scriptscriptstyle #1}
                \accentadjscriptscriptback{\scriptscriptstyle #1}}$}{#1}}}}
\def\itunderline#1{{\mathsurround0pt\mathchoice
        {\rlap{$\underline{\phantom{\displaystyle #1}
                \accentadjtextback{\displaystyle #1}}$}{#1}}
        {\rlap{$\underline{\phantom{\textstyle #1}
                \accentadjtextback{\textstyle #1}}$}{#1}}
        {\rlap{$\underline{\phantom{\scriptstyle #1}
                \accentadjscriptback{\scriptstyle #1}}$}{#1}}
        {\rlap{$\underline{\phantom{\scriptscriptstyle #1}
                \accentadjscriptscriptback{\scriptscriptstyle #1}}$}{#1}}}}
\newcommand{\ga}{\gamma}

\newcommand{\eps}{\varepsilon}

\newcommand{\Om}{\Omega}

\newcommand{\clOmP}{{\overline{\Om}\mspace{1mu}}^P}
\newcommand{\CapPp}{{\overline{C}\mspace{1mu}}^P_{p}}

\newcommand{\bdP}{{\bdy_{\rm P}}}
\newcommand{\bdySP}{\partial_{\rm SP}}
\newcommand{\bdyRSP}{\partial_{\rm RSP}}

\newcommand{\Irr}{{\rm Irr}}
\renewcommand{\phi}{\varphi}
\newcommand{\p}{{$p\mspace{1mu}$}}
\newcommand{\R}{\mathbb{R}}

\newcommand{\N}{\mathbb{N}}
\newcommand{\eR}{{\overline{\R}}}

 \newcommand{\distMa}{{\rm dist}_{\rm{M}}}
\newcommand{\limminus}{{\mathchoice{\raise.17ex\hbox{$\scriptstyle -$}}
                {\raise.17ex\hbox{$\scriptstyle -$}}
                {\raise.1ex\hbox{$\scriptscriptstyle -$}}
                {\scriptscriptstyle -}}}
\newcommand{\limplus}{{\mathchoice{\raise.17ex\hbox{$\scriptstyle +$}}
                {\raise.17ex\hbox{$\scriptstyle +$}}
                {\raise.1ex\hbox{$\scriptscriptstyle +$}}
                {\scriptscriptstyle +}}}
\newcommand{\Np}{N^{1,p}}
\newcommand{\Npz}{N_{0}^{1,p}}

\newcommand{\uHp}{\itoverline{P}}   
\newcommand{\lHp}{\itunderline{P}}  
\newcommand{\A}{\mathcal{A}}%
\newcommand{\UU}{\mathcal{U}}%
\newcommand{\LL}{\mathcal{L}}%
\newcommand{\g}{\gamma}

\newcommand{\dM}{d_M}

\newcounter{komcounter}
\numberwithin{komcounter}{section}

\makeatletter
\newcommand{\setcurrentlabel}[1]{\def\@currentlabel{#1}}
\makeatother

\numberwithin{equation}{section}

\newenvironment{ack}{\medskip{\it Acknowledgement.}}{}

\begin{document}

\authortitle{T.~Adamowicz
and N.~Shanmugalingam}
{Prime end capacity of inaccessible prime ends, resolutivity, Kellogg property}
\title{The prime end capacity of inaccessible prime ends, resolutivity, and the Kellogg property}
\author{Tomasz Adamowicz\small{$^1$}
\\
\it \small The Institute of Mathematics, Polish Academy of Sciences,\\
\it \small \'Sniadeckich 8, Warsaw, 00-656, Poland\/{\rm ;}
\it \small tadamowi@impan.pl\\
\\
Nageswari Shanmugalingam\small{$^2$}
\\
\it \small  Department of Mathematical Sciences, University of Cincinnati, \\
\it \small  P.O.\ Box 210025, Cincinnati, OH 45221-0025, U.S.A.\/{\rm ;}
\it \small shanmun@uc.edu
}

\date{}
\maketitle

\footnotetext[1]{T. Adamowicz was supported by a grant Iuventus Plus of the Ministry of Science and Higher 
Education of the Republic of Poland, Nr 0009/IP3/2015/73.}
\footnotetext[2]{N. Shanmugalingam's research was partially supported by the NSF (U.S.A.) grant DMS~\#1500440.}

\noindent{\small
{\bf Abstract}. Prime end boundaries $\bdP\Om$ of domains $\Om$ are studied in the setting of complete doubling metric measure 
spaces supporting a $p$-Poincar\'e inequality. Notions of rectifiably (in)accessible- and (in)finitely far away prime 
ends are introduced and employed in classification of prime ends.  We show that, for a given domain, the prime 
end capacity (defined in~\cite{ES}) of the collection of all rectifiably inaccessible prime ends together will all
non-singleton prime ends is zero.  We show the resolutivity of 
continouous functions on $\bdP\Om$ which are Lipschitz
continuous with respect to the Mazurkiewicz metric when restricted to the collection $\bdySP\Om$ of all  
accessible prime ends. Furthermore, bounded perturbations of such functions in $\bdP\Om\setminus\bdySP\Om$
yield the same Perron solution. 
In the final part of the paper, we demonstrate
the (resolutive) Kellogg property with respect to the prime end boundary of bounded domains in the metric space.
Notions given in this paper are illustrated by a number of examples.
}

\bigskip
\noindent
{\small \emph{Key words and phrases}: capacity, doubling measure, Dirichlet problem, Ma\-zur\-kie\-wicz distance,  
metric measure spaces, $p$-harmonic functions, Perron method, Poincar\'e inequality, prime end boundary, resolutivity,
Kellogg property.
}

\medskip
\noindent
{\small Mathematics Subject Classification (2010): Primary: 31E05; Secondary: 31B15, 31B25, 31C15, 30L99
}

\section{Introduction}

  In the classical Dirichlet problem, given a differential operator $L$, one asks whether 
  there is a solution $u$ to the equation $Lu=0$ in $\Om$ satisfying a prescribed boundary condition $u=f$ on 
  $\partial \Om$. Here, for continuous functions $f:\partial \Om\to\R$ 
  we say that $u=f$ on $\partial\Om$ if for almost every (in a capacitary sense) $x\in\partial\Om$ we have
  \begin{equation}\label{eq:strong-bdry}
   \lim_{\Om\ni y\to x}u(y)=f(x).
  \end{equation}
  This problem has been investigated for various types of PDEs and settings, including the Laplace 
  equation and its nonlinear counterpart, the $p$-Laplace equation. 
  In this paper we will focus on application to the $p$-harmonic equation, but the prime end boundary 
  theory is accessible for a wider class of PDEs.
  In the Euclidean setting of a domain
  $\Om\subset \R^n$ the $p$-Laplace equation reads:
\[ 
 {\rm div}(|\nabla u|^{p-2}\nabla u)=0,
\]
where $u$ belongs to the Sobolev space $W^{1,p}_{loc}(\Om)$ and $1<p<\infty$. Moreover, 
for more general functions $f:\Om\to\R$ such that $f\in W^{1,p}_{loc}(\Om)$ the Dirichlet problem is 
understood in the weak sense, i.e. $u-f\in W^{1,p}_{0}(\Om, \R)$. However, the 
boundary value problem studied with respect to the topological boundary of the domain is often too restrictive 
and the corresponding solution does not fully capture the geometry of the domain. This is the case of the planar 
disc with a radial line removed (the so-called slit disc), the topologists' comb (and its higher dimensional 
generalizations) and many other domains with nontrivial boundary, see~\cite{BjComb, bbs1}. For example,
in the case of the slit disc, every point of the slit (except for its tip) has prescribed a single boundary value, 
even though it might be more desirable to have different prescribed behavior of the solution when approaching
the slit from above and from below. In this situation, 
it would be more natural to associate two boundary points to each point on the slit. 
Boundaries of such domains are described more effectively by other forms than just the topological boundary.
There are several notions of abstract boundaries, eg.~the 
Martin, Royden, and prime end boundaries. The prime end boundaries, as considered in~\cite{ES}, form the focus of this paper.  
  
 A formulation of prime ends was initiated by Carathe\'odory in 1913 for simply-connected planar domains in the 
 setting of boundary extension of conformal mappings, see Section~\ref{sec:pe} below for more information and a 
 brief historical account of prime ends. Carathe\'odory's construction of prime ends is not productive for multiply connected
 planar domains and for more general domains in metric measure spaces.
 Here we study prime ends in the context of bounded domains in complete metric measure spaces 
 equipped with a doubling measure and supporting a $p$-Poincar\'e inequality. This set of assumptions allows for a 
 viable first order Calculus, in particular various counterparts of the Sobolev spaces are available, see~\cite{he07}. 
 Furthermore, in this setting a variational analog of the $p$-Laplacian is available and the nonlinear potential theory 
 for $p$-harmonic functions in metric measure spaces is well developed, see~\cite{BBbook, bbs1, BBS2}. 
 
 Two of the cornerstones of the Perron method are (a)~that every continuous boundary data is resolutive and
 (b)~the \emph{Kellogg property}. The Kellogg property is that there
 is a subset of the topological boundary of the given domain, with zero $p$-capacity, 
 such that~\eqref{eq:strong-bdry} is satisfied by each continuous data on the topological boundary $\partial\Om$
 outside of this subset; see~\cite{Hed, HeKiMa} for the Euclidean setting, 
 and~\cite{BBS} for a proof of the Kellogg property in the metric setting. 
 Thus $p$-capacitary almost every point on the boundary of the domain influences the behavior of 
 Perron solutions. 
 
 In the case of the double comb (see Example~\ref{ex2}), the one-dimensional almost
 every point in the interval $[1/4,3/4]\times\{0\}$ will influence the behavior of the solution, and modifying the boundary
 data on this subset of the (topological) boundary might result in a different Perron solution. 
 As a consequence of one of the main results of this paper, Theorem~\ref{prop:exceptional},
 we know that for Perron solutions constructed using the prime end boundary instead of the topological boundary,
 the subset $[1/4,3/4]\times\{0\}$ will not affect the solution; this confirms the results of~\cite{BjComb}. 
 In Theorem~\ref{prop:exceptional} we identify a subset of prime ends 
 whose prime end capacity is zero and do not influence the Perron solution. 
 
 It turns out that the key to understanding which prime ends are vital for the solvability of the Dirichlet problem lies 
 in accesibility of points in the prime end impressions through rectifiable curves. This approach 
 allows us to classify prime ends into four categories as: 
(a)~\emph{rectifiably accessible},
(b)~\emph{rectifiably inacessible},
(c)~\emph{infinitely far away}, and
(d)~\emph{finitely away}.  The first two categories of prime ends belong to the class $\bdySP\Om$ of singleton prime ends, and the latter two belong to the class of non-singleton prime ends.
 Similar classifications are known in the theory 
 of prime ends of Carath\'eodory in $\R^2$, see \cite[Chapter 9]{cl} and the prime ends constructed by 
 N\"akki in $\R^n$, see~\cite[Section 8]{na}. 
 The Kellogg property for the  prime end boundary of domains whose prime end boundary consists solely of
 singleton prime ends has been studied in~\cite{Bj17}, where the property is verified for resolutive continuous functions
 on the boundary, and such a property is called the \emph{resolutive Kellogg property} in~\cite{Bj17}. In this paper
 we extend this property to domains whose prime end boundary might have more than just singleton prime ends.
 The proof in~\cite{Bj17} crucially uses the compactness of the singleton prime end boundary of the domain, and if
 the domain has more than just singleton prime ends in its prime end boundary, such compactness \emph{must} fail.
 Hence we do not follow the method in~\cite{Bj17} but go back to the basics of the argument found in~\cite{BBS}.

  The organization of the paper is as follows. In Section~\ref{sect-New-Sob} we recall some basic notions of 
  analysis on metric measure spaces. Section~\ref{sec:pe} is devoted to recalling the construction of prime ends 
  for domains in metric spaces. We define, and illustrate with examples,  
  notions of rectifiably (in)accessible prime ends, see Definition~\ref{SP-R+NR} and (in)finitely 
  far away prime ends, see Definition~\ref{NSP-R+NR}. The prime end capacity $\CapPp$ and the Perron method are 
  discussed in Section~2.3, while in Section~3 we show that the  non-singleton prime ends and (singleton) 
  rectifiably inaccessible prime ends, together, form a prime end capacity null set. This is the content of the first principal result
  of this paper, Theorem~\ref{prop:exceptional}.
These observations allow us to answer an open question posed in~\cite{ES}, see Remark~3.4. In Section~4,
using the results from Section~3, we prove the second principal result of this paper,
Theorem~\ref{thm:main}. There, we show the resolutivity of functions defined on the prime end 
boundary $\bdP \Om$ of the given domain $\Om$, which are Lipschitz continuous with respect to the 
Mazurkiewicz distance when restricted to the part of $\bdP \Om$ consisting of the rectifiably accessible 
prime ends only. This improves resolutivity results presented in~\cite{ES}, see Remark~4.2 below for 
detailed discussion. See also Proposition~\ref{prop:ext} for further extension of Theorem~\ref{thm:main}. 
Section~5 contains five examples illustrating the features of results obtained in Section~4. We prove
the (resolutive) Kellogg property in Section~6 by showing that there is a set $\Irr(\Om)\subset\bdP\Om$
with $\CapPp(\Irr(\Om))=0$ such that the Perron solution of every resolutive continuous boundary data
on $\bdP\Om$ achieves the correct limiting behavior~\eqref{eq:strong-bdry} 
in $\bdP\Om\setminus\Irr(\Om)$, see Theorem~\ref{thm:Kellogg}.
  
\begin{ack}
Part of the research for this paper was conducted during the second author's visit to IM PAN, Poland, in July~2016
and July~2018, and to Link\"oping University in Spring~2018; she thanks these august institutions for their kind hospitality.
\end{ack}

\section{Notation and preliminaries}

In this section we provide descriptions of the basic notions used in the paper. We recommend interested readers
to look to the books~\cite{BBbook, HKST} and the papers~\cite{abbs, ES, bbs1, BjComb, Bj17} 
for more information pertaining to these notions.

\subsection{Newton-Sobolev spaces}\label{sect-New-Sob}
Let $(X, d, \mu)$ be a complete metric measure space equipped with a metric $d$ and a doubling measure $\mu$.
Recall that $\mu$ is doubling if $\mu$ is a Radon measure and
there is some $C\ge 1$ such that whenever $x\in X$ and $r>0$, we have
\[
0<\mu(B(x,2r))\le C\, \mu(B(x,r))<\infty,
\]
where, $B(x,r)=\{y\in X\, :\, d(y,x)<r\}$.
As $\mu$ is doubling and $X$ is complete, necessarily $X$ is proper, i.e. closed and bounded subsets of $X$ are compact. 
A curve in $X$ is a continuous mapping $\ga:[a, b]\to X$. The image of $\ga$ (locus/trajectory) is 
denoted by $|\ga|=\ga([a,b])$. The length of $\ga$ is denoted by $\ell(\ga)$ and we say that $\ga$ is rectifiable if 
$\ell(\ga)<\infty$. Every rectifiable curve admits the so-called arc-length parametrization, see e.g.~\cite[Section~5.1]{HKST}
or~\cite[Section~4.2]{AT}.  See~\cite[Section~5.1]{HKST} for discussions related to integrals 
$\int_\gamma g\, ds$ of Borel functions $g$ on $X$
along rectifiable paths $\gamma$.

 Next, we recall the basic notions in the theory of first order calculus in metric measure spaces.
 We say that a nonnegative Borel  function $g$ on $X$ is an \emph{upper gradient} of a function 
 $u:X\to[-\infty,\infty]$, if for each nonconstant rectifiable curve $\ga: [0,\ell(\ga)]\to X$ we have
 \begin{equation*}
    |u(x)-u(y)|\le \int_\gamma g\, ds
 \end{equation*}
where $x$ and $y$ are the two endpoints of $\ga$. If at least one of $|u(x)|$, $|u(y)|$ is infinite, then we interpret 
the above inequality to mean that  
$\int_\g g\, ds=\infty$. 
Recall that a family of rectifiable curves in $X$ is of zero $p$-modulus if there is a non-negative Borel measurable
function $g\in L^p(X)$ such that $\int_\gamma g\, ds=\infty$ for each $\gamma\in \Gamma$.
We say that $g$ is a \emph{$p$-weak upper gradient} of $u$ if the collection $\Gamma$ of
curves for which the above inequality fails is of $p$-modulus zero.

Upper gradients were introduced by Heinonen and Koskela in~\cite{HeKo98}. It is easy to notice that if $g$ is an upper 
gradient, then so is $g+h$ for any nonnegative Borel function $h$ on $X$. The more handy unique gradient to work 
with is the so called \emph{minimal $p$-weak upper gradient} $g_u \in L^p(X)$, which is the $p$-weak upper gradient
with smallest $L^p$-norm, see e.g~the discussion in~\cite{HKST}. 

The following version of Sobolev spaces on the metric space $X$ will be considered in this paper;
see~\cite{Sh-rev,BBbook,HKST} for more on this space. For $u:X \to [-\infty,\infty]$ a measurable function, set
\[
    \|u\|_{\Np(X)} := \left( \int_X |u|^p \, d\mu + \inf_g  \int_X g^p \, d\mu \right)^{\frac1p},
\]
where the infimum is taken over all upper gradients $g$ of $u$ (or equivalently, over all $p$-weak upper gradients $g$ of $u$). 
With this notation we define the \emph{Newtonian space} 
on $X$ as follows:
\[
        \Np (X) = \{u: \|u\|_{\Np(X)} <\infty \}/\sim,
\]
where functions $u$ and $v$ are equivalent, denoted $u\sim v$, if $\|u-v\|_{\Np(X)}=0$. 

The \emph{Sobolev $p$-capacity} of a set $E\subset X$ is defined as follows:
\begin{equation}
  C_p (E) :=\inf \|u\|_{N^{1, p}(X)}^p, \label{Sob-cap}
\end{equation}
where the infimum is taken over all functions $u\in N^{1, p}(X)$ that have a representative, also denoted $u$,
such that $u \geq 1$ on $E$ (see 
e.g.~\cite[Chapter 1.4]{BBbook} and~\cite[Chapter 7.2]{HKST} for definitions and properties of the Sobolev capacity). 
This capacity measures the exceptionality of sets in the potential theory related to Newtonian spaces and is a finer 
way to detect smallness of sets than null $\mu$-measure.
\vspace{0.3cm}

\begin{deff}\label{def:p-PI}
We say that $X$ supports a \emph{$p$-Poincar\'e inequality} if there exist constants $C>0$ and $\lambda \ge 1$
such that for all balls $B \subset X$ and all $u\in \Np (X)$,
\begin{equation*} 
        \vint_{B} |u-u_B| \,d\mu  \le C\, \diam(B) \left( \vint_{\lambda B} g_u^{p} \,d\mu \right)^{\frac1p},
\end{equation*}
where $u_B$ stands for the mean-value of $u$ on $B$:
\[
u_B:=\vint_B u \, d\mu :=\frac{1}{\mu(B)}\int_B u \,d\mu.
\]
\end{deff}

\begin{deff}\label{def:zero-boundary}
Given a domain $\Om$ in $X$ we denote by $\Npz(\Om)$ the space of Newtonian functions with zero boundary data; these are
functions $f\in \Np(X)$ such that $f(x)=0$ for $p$-capacity almost every point  $x\in X\setminus\Om$.
\end{deff}

 In addition to the assumptions outlined at the beginning of this section, we 
 will also assume in this paper that $X$ supports the $p$-Poincar\'e inequality for a fixed $1\leq p<\infty$ (see 
 below). This together with doubling measure $\mu$ implies that $X$ is \emph{quasiconvex}, meaning that there is a 
 constant $C_q\geq 1$ such that for any points $x,y\in X$ there is a rectifiable curve $\ga$ joining $x$ and $y$ in $X$ 
 satisfying $\ell(\ga)\leq C_q d(x,y)$.

\subsection{\bf Prime ends in metric spaces}\label{sec:pe}

\noindent We now turn our attention to the main object of our work, namely prime ends and the prime end boundary.

The first theory of prime ends is due to Carath\'eodory, who formulated a definition of prime ends from the point of view of 
conformal mappings in simply-connected 
planar domains. Subsequently, the theory has extended to more general domains in the plane and 
in higher dimensional Euclidean spaces, see for instance the works of Freudenthal, Kaufman, Mazurkiewicz, and 
more recently Epstein and N\"akki (for further discussion of the history of prime ends and the literature, we refer 
to~\cite[Sections~1,3]{abbs}, see also~\cite{aw} for an application of the prime end theory in the setting of Heisenberg groups). Here we study prime ends in the more general setting of metric spaces, 
see~\cite{abbs, Est, ES}. The notion of prime ends considered here is from~\cite{ES}, and is a slight modification
from that of~\cite{abbs}.  First, we recall the notion of the Mazurkiewicz distance.

\begin{deff} \label{Mazur}
Given a domain $\Om\subset X$, the \emph{Mazurkiewicz metric} 
$d_M$ on $\Om$ is given by
\[
 d_M(x,y)=\inf_E \diam(E)
\]
for $x,y\in\Om$, where the infimum is over all connected compact subsets $E\subset\Om$ containing $x$ and $y$.
\end{deff}

We will assume throughout this paper that the measure on $X$ is doubling and supports a $p$-Poincar\'e inequality. It 
follows that $X$ is quasiconvex, and so in $\Om$ we know that both $d_M$ and $d$ are locally biLipschitz equivalent.
Furthermore, as $\Om$ is connected, $d_M$ is indeed a metric on $\Om$.

Given two sets $A,K\subset \Om$, we set
\begin{align*}
\text{dist}(A,K)&:=\inf\{d(x,y)\, :\, x\in A,\, y\in K\},\\
\distMa(A,K)&:=\inf\{d_M(x,y)\, :\, x\in A,\, y\in K\}.
\end{align*}

Let $\Om\subsetneq X$ be a bounded domain in $X$, i.e. a bounded nonempty connected open subset 
of $X$ that is not the whole space $X$ itself.
A connected set $E\subsetneq\Om$ is called an \emph{acceptable} set if 
$\overline{E}\cap \partial \Omega\not=\emptyset$. The boundedness and connectedness of an 
acceptable set $E$ implies that $\overline{E}$ is compact and connected, that is, 
$\overline{E}$ is a continuum.

\begin{deff}\label{def-chain}
A sequence $\{E_k\}_{k=1}^\infty$ of acceptable sets is called a \emph{chain} if the following conditions are satisfied
for each $k\in\N$:
\begin{enumerate}
\item \label{it-subset}
$E_{k+1}\subset E_k$,
\item \label{pos-dist}
$\distMa(\Omega\cap\bdy E_{k+1},\Omega\cap \bdy E_k )>0$,
\item \label{impr}
The \emph{impression} $\bigcap_{k=1}^\infty \overline{E}_k \subset \bdy\Om$.
\end{enumerate}
\end{deff}

Note that the impression is either a point or a continuum, since $\{\overline{E}_k\}_{k=1}^\infty$ is a 
decreasing sequence of continua. 

The above definition of a chain differs from the definition of prime ends in~\cite[Definition~4.2]{abbs} only in that
we require $\distMa(\Omega\cap\bdy E_{k+1},\Omega\cap \bdy E_k )>0$ in
condition~\ref{pos-dist}  rather than that
$\dist(\Omega\cap\bdy E_{k+1},\Omega\cap \bdy E_k )>0$.
However, we emphasize that such a modification does not affect results 
from~\cite{abbs}, see~\cite[Definition~2.3]{ES} and the discussion following it for comparison between the above 
definition and~\cite[Definition~4.2]{abbs}. In general, there are more chains and ends in the sense of the above 
definition than in the sense of~\cite{abbs}, and therefore a priori a prime end in the setting of~\cite{abbs} need not be 
prime in our sense. We further note that results in~\cite{abbs} employed below, which use the analog of 
condition~\ref{pos-dist} in Definition~\ref{def-chain} for $d$ instead of $\distMa$ are, in fact, 
based on the positivity of the Mazurkiewicz distance; hence the results of~\cite{abbs} do apply here. 

\begin{deff}
We say that a chain $\{E_k\}_{k=1}^\infty$ \emph{divides} the chain $\{F_k\}_{k=1}^\infty$ if for each $k\in\N$ there 
exists $l_k\in\N$ such that $E_{l_k}\subset F_k$. We say that two chains 
are \emph{equivalent} if they divide each other. The collection of all chains that are equivalent to a given chain
$\{E_k\}_{k=1}^\infty$ is called 
an \emph{end} and is denoted $[E_k]$. The 
\emph{impression of} an end $[E_k]$, denoted $I[E_k]$, is defined as the impression of any representative chain.
\end{deff}

The impression of an end is independent of the choice of representative chain, see \cite[Section 4]{abbs}. Note 
also that if a chain $\{F_k\}_{k=1}^\infty$ divides $\{E_k\}_{k=1}^\infty$, then it divides every chain equivalent to 
$\{E_k\}_{k=1}^\infty$. Moreover, if $\{F_k\}_{k=1}^\infty$ divides $\{E_k\}_{k=1}^\infty$, then every chain 
equivalent to $\{F_k\}_{k=1}^\infty$ also divides $\{E_k\}_{k=1}^\infty$. Therefore, the relation of division extends 
in a natural way from chains to ends, defining a partial order on ends.

\begin{deff}\label{prime-end}
We say that an end $[E_k]$ is a \emph{prime end} if it is not divisible by any other end. The collection of 
all prime ends is called the \emph{prime end boundary} and is denoted $\bdP\Omega$.
The collection of all prime ends with singleton impressions is denoted $\bdySP \Omega$.
\end{deff}

\begin{remk}\label{rem:SP}
It is not difficult to see that an end with a singleton impression is necessarily a prime end. Moreover,
if $\gamma:[a,b]\to\overline{\Om}$ is a curve such that $\gamma([a,b))\subset\Om$, $\gamma(b)\in\partial\Om$,
and for each $m\in\N$ there exists $t_m\in(a,b)$ such that $\gamma((t_m,b))\subset E_m$, then
$[E_m]$ is a singleton prime end with  impression $\{\gamma(b)\}$, see~\cite{abbs}.
It was shown in~\cite[Theorem~9.6]{abbs} that $d_M$ extends as a metric to
$\bdySP\Om$ and that $\Om\cup\bdySP\Om$ is complete under this metric, but not
necessarily compact. Furthermore, $\Om\cup\bdySP\Om$ is compact with respect to $d_M$ if and only
if $\bdP\Om=\bdySP\Om$, or equivalently, $\Om$ is finitely connected at the boundary.
\end{remk}

In order to set up a viable topology on the prime ends closure $\overline{\Om}^P:=\Om\cup\bdP \Om$ let us recall
the following notion, cf.~\cite[Section~8]{abbs}. We say that a \emph{sequence of points} 
$\{x_n\}_{n=1}^\infty$ in $\Omega$ \emph{converges to the prime end} $[E_k]$, 
and write $x_n \to [E_k]$ as $n \to \infty$, if for all $k\in\N$ there exists $n_k\in\N$ such that 
$x_n\in E_k$ whenever $n \ge n_k$.
Next, we define the sequential topology on $\clOmP$.

\begin{deff}\label{PrimeEndTop}
Given a sequence $\{p_k\}_k\in\clOmP$ and $[E_k]\in\bdP\Om$, we say that 
$\lim_kp_k=[E_k]$ if the subsequence of $\{p_k\}_k$ that is from $\Om$ converges to $[E_k]$ in the 
above sense, and for the subsequence (if any) of prime ends, denoted without loss of generality by
$\{[F_{j,k}]\}_j$, we have that for each $k\in\N$ there is some $j_k\in\N$ such that whenever 
$j\in\N$ with $j\ge j_k$, the prime end $[F_{j,k}]$ satisfies the condition that for
any/each respresentative chain $\{F_{j,m}\}_m$  we have $F_{j,m}\subset E_k$ for large enough $m$.
It is possible for $\{p_k\}_k$ to converge to two distinct prime ends, see the discussion in~\cite{abbs}.
This notion of convergence induces a topology on $\clOmP$, called the prime end topology of $\clOmP$. 
\end{deff}

Such a 
topology satisfies the separation condition ($T1$), but need not be, in general, Hausdorff ($T2$), 
see~\cite[Proposition 8.8, Example 8.9]{abbs}.
A basis for the topology on $\clOmP$ is defined as follows. Given $G\subset \Om$, we let
\[
 G^P:=G\cup \{[E_k]\in \bdP \Om\,:\, \text{ there exists }n\in\N\text{ with }E_n\subset G\}.
\]
Then, \cite[Proposition 8.5]{abbs} shows that the following collection 
of sets forms a basis for the topology on $\clOmP$:
\[
 \{G, G^P\,:\, G\subset \Om \hbox{ is open} \}.
\]

In addition to our afore-mentioned assumptions on space $X$, in what follows we will also require the domain 
$\Om$ to satisfy the following assumption (cf. Assumption~4.7 in~\cite{ES} and the discussion therein):

\begin{assumption}\label{assumption-ends}
 For every collection $\mathcal{E}$ of ends in $\Om$ that is totally ordered by division (i.e.: for $x, y\in \mathcal{E}$ 
 we define that $x\leq y$ if and only if $x$ divides $y$),  there \emph{must} exist an end $[F_k]$ such that $[F_k]\leq [E_n]$ 
 for every $[E_n]\in \mathcal{E}$. 
\end{assumption}

In other words, we assume that the collection of all ends in $\Om$ satisfies the hypotheses of the Kuratowski--Zorn 
lemma. The above assumption is satisfied, for instance, if $\Om$ is a simply-connected bounded domain in $\R^2$ 
or if $\bdySP \Om$ is compact (and hence $\bdP \Om=\bdySP \Om$), see the details of discussion on 
pg.~346 in~\cite{ES}. We do not know of an example where this assumption fails.
We employ this assumption when using 
Theorems~\ref{thm77} and~\ref{thm78} (Theorems 7.7 and 7.8 in~\cite{ES}) in Section~\ref{sect:res}. 
The proofs of these results rely on the comparison 
principle (\cite[Proposition~7.3]{ES}) and~\cite[Corollary~5.4]{ES}, which in turn depend on 
Assumption~\ref{assumption-ends}.

We now classify prime ends according to whether they allow approach to the impression, from inside the chain
that makes up the end, along a rectifiable curve.

\begin{deff} \label{SP-R+NR}
A prime end $[E_k]\in \bdySP\Om$ is said to be  \emph{rectifiably accessible} if there is a rectifiable curve
$\gamma:[0,1]\to X$ such that $\gamma([0,1))\subset\Omega$, $\{\gamma(1)\}= I[E_k]$, and for each $k\in\N$ there is
some $t_k\in(0,1)$ such that $\gamma((t_k,1))\subset E_k$. We say that $[E_k]\in\bdySP\Om$  is
\emph{rectifiably inaccessible} if it is not rectifiably accessible.
\end{deff}

\begin{example}(see  Figure 1)\label{ex1}
The following domain provides us with an example of a rectifiably inaccessible singleton prime end. Consider the
graph of the function $f:[0,1]\to\R$ given by a ``harmonically damped sawtooth", that is, 
\[
f(x)=\begin{cases} 
	2(n+1)\left[x-\tfrac{2n+1}{2n(n+1)}\right]&\text{ if }x\in[\tfrac{2n+1}{2n(n+1)},\tfrac{1}{n}],\\
      -2n\left[x-\tfrac{2n+1}{2n(n+1)}\right]&\text{ if }  x\in[\tfrac{1}{n+1}, \tfrac{2n+1}{2n(n+1)}] \end{cases}
\]
for each positive integer $n$. Let $\Om\subset\R^2$ be obtained by 
\[
\Om:=\bigcup_{n=1}^{\infty}\,\bigcup_{x\in \left(\frac{1}{n+1}, \frac1n\right]} \{x\}\times(f(x)-n^{-4}, f(x)+n^{-4}).
\]
For this domain, the point $(0,0)$ is the impression of exactly one prime end from $\bdySP\Om$. 
Indeed, a chain is given by a sequence of 
balls centered at point $(0,0)$ with shrinking radii, intersected with $\Om$. Since the resulting end has the
singleton impression $\{(0,0)\}$, it is a prime end, by Proposition 7.1 in~\cite{abbs}. Moreover, this prime 
end is rectifiably inaccessible, as each curve in $\Om$ with the end point $(0,0)$ should have length at 
least $\sum_{n=k}^\infty \frac{1}{n}=\infty$, see  Figure~1.
\definecolor{zzttqq}{rgb}{0.6,0.2,0.}
\definecolor{uuuuuu}{rgb}{0.26666666666666666,0.26666666666666666,0.26666666666666666}
\begin{figure}[t]
\centering
\subfloat[\empty]
{
\begin{tikzpicture}[line cap=round,line join=round,>=triangle 45, x=2.3cm,y=1.5cm]
\draw[->,color=black] (-0.6557514812466693,0.) -- (3.2506136827385212,0.);
\foreach \x in {-0.6,-0.4,-0.2,0.2,0.4,0.6,0.8,1.,1.2,1.4,1.6,1.8,2.,2.2,2.4,2.6,2.8,3.,3.2}
\draw[shift={(\x,0)},color=black] (0pt,-2pt);
\draw[color=black] (3.1831170924320817,0.03983783753066676) node [anchor=south west] { x};
\draw[->,color=black] (0.23358515970077589,-2.3652605615412288) -- (0.23358515970077589,4.158185334105454);
\foreach \y in {-2.,-1.5,-1.,-0.5,0.5,1.,1.5,2.,2.5,3.,3.5,4.}
\draw[shift={(0,\y)},color=black] (-2pt,0pt);
\draw[color=black] (0.261092684470762367,3.9390772276867865) node [anchor=west] { y};
\clip(-0.6557514812466693,-2.3652605615412288) rectangle (3.2506136827385212,4.158185334105454);
\fill[line width=2.pt,color=zzttqq,fill=zzttqq,fill opacity=0.3499999940395355] (2.406001401405573,3.594639666974907) -- (2.399121644696401,0.) -- (1.8008091879679622,-1.8418327588205663) -- (1.8008091879679622,1.7995780169481) -- cycle;
\fill[line width=2.pt,color=zzttqq,fill=zzttqq,fill opacity=0.3499999940395355] (1.8008091879679622,1.7995780169481002) -- (1.2061407021510402,2.99943871620263) -- (1.2163981972940507,-0.9291819235703255) -- (1.8008091879679622,-1.8418327588205663) -- cycle;
\fill[line width=0.4pt,dash pattern=on 1pt off 1pt,color=zzttqq,fill=zzttqq,fill opacity=0.3499999940395355] (1.2092725156680428,1.7999541391906626) -- (0.9870387766776816,0.5957637500421277) -- (0.7889716060507951,1.4155417618034083) -- (0.8,0.2) -- (1.0090709023283804,-0.6032724142090399) -- (1.2122317611091642,0.6665631352412253) -- cycle;
\fill[line width=2.pt,color=zzttqq,fill=zzttqq,fill opacity=0.3499999940395355] (0.48073306015792455,0.5460354183937188) -- (0.43026046184970224,0.04707773226101074) -- (0.39565068015263555,0.3297242827870535) -- (0.39420860591525775,0.16965404243812113) -- (0.4345866845618356,-0.04665709316854425) -- (0.4850592828700579,0.3816389553326532) -- cycle;
\fill[line width=2.pt,color=zzttqq,fill=zzttqq,fill opacity=0.3499999940395355] (0.39456756512938806,0.25319119833634446) -- (0.37102744645506447,0.016166555132811636) -- (0.3523576971616354,0.14604307195666527) -- (0.3539811536219336,0.0892220958462293) -- (0.3726509029153627,-0.018737758763599018) -- (0.39456756512938806,0.19718195045605763) -- cycle;
\fill[line width=2.pt,color=zzttqq,fill=zzttqq,fill opacity=0.3499999940395355] (0.7915894847786129,1.1495899428435203) -- (0.6880124751163107,0.28921811912096) -- (0.5917296210640297,0.9537590426678629) -- (0.6077767634060766,0.14796304358442025) -- (0.701141955214349,-0.3592710912967907) -- (0.79742480926663,0.5203627881807523) -- cycle;
\fill[line width=2.pt,color=zzttqq,fill=zzttqq,fill opacity=0.3499999940395355] (0.5969207611789771,0.7770111367751413) -- (0.5544894992104802,0.17793182593718482) -- (0.47817038991445987,0.6318503795095741) -- (0.48642110443294856,0.3322641341517972) -- (0.5668655709882132,-0.17158546031355498) -- (0.6025082626225347,0.402334740523864) -- cycle;
\draw [line width=1.2pt] (1.8165845488808996,0.)-- (1.20015982870885,1.193112453325547);
\draw [line width=1.2pt] (1.20015982870885,1.193112453325547)-- (1.,0.);
\draw [line width=1.2pt] (1.,0.)-- (0.7892419961830267,0.8056107579699517);
\draw [line width=1.2pt] (0.7892419961830267,0.8056107579699517)-- (0.6876366887733469,0.);
\draw [line width=1.2pt] (0.6876366887733469,0.)-- (0.5973208599647427,0.6136896217516676);
\draw [line width=1.2pt] (0.5973208599647427,0.6136896217516674)-- (0.5609622334714767,0.);
\draw [line width=1.2pt] (0.5609622334714767,0.)-- (0.4800841062183786,0.4819102852073083);
\draw (2.3773765456489593,-0.04470652537988971) node[anchor=north west] {1};
\draw (1.1244710880856745,1.6085637321427808) node[anchor=north west] {$\frac{1}{2}$};
\draw (1.1624379201330468,0.07480698721211057) node[anchor=north west] {$\frac{1}{2}$};
\draw (0.7658954520827143,0.08476644659477726) node[anchor=north west] {$\frac{1}{3}$};
\draw (0.7490213045061044,1.3296988694281136) node[anchor=north west] {$\frac{1}{3}$};
\draw (0.5296573860101758,0.09472590597744396) node[anchor=north west] {$\frac{1}{4}$};
\draw (0.550750070480938,1.1504286005401132) node[anchor=north west] {$\frac{1}{4}$};
\draw (0.41153835297390645,1.0010367098001127) node[anchor=north west] {$\frac{1}{5}$};
\draw (0.10358515970077589,-0.16422003797189) node[anchor=north west] {(0,0)};
\draw [line width=1.2pt] (2.405976179779649,1.7991975907570856)-- (1.8165845488808994,0.);
\draw (2.3984692301197215,1.9671042699187817) node[anchor=north west] {1};
\draw (1.8163111387266804,1.4715637378254458) node[anchor=north west] {f(x)};
\begin{scriptsize}
\draw [fill=black] (2.399121644696401,0.) circle (0.5pt);
\draw [fill=black] (1.20015982870885,1.193112453325547) circle (0.5pt);
\draw [fill=black] (1.1986385919234932,0.) circle (0.5pt);
\draw [fill=black] (0.799354761924179,0.) circle (0.5pt);
\draw [fill=black] (0.7892419961830266,0.8056107579699516) circle (0.5pt);
\draw [fill=black] (0.5936630919245324,0.) circle (0.5pt);
\draw [fill=black] (0.5973208599647427,0.6136896217516674) circle (0.5pt);
\draw [fill=black] (0.4800841062183786,0.4819102852073083) circle (0.5pt);
\draw [fill=black] (2.405976179779649,1.7991975907570856) circle (0.5pt);
\draw [fill=uuuuuu] (1.8165845488808994,0.) circle (0.5pt);
\draw [fill=black] (0.24701541410195998,-0.0048679751038052345) circle (1.5pt);
\end{scriptsize}

\end{tikzpicture}
}
\caption{Example~\ref{ex1}.}
\end{figure}
\end{example}

\begin{deff} \label{NSP-R+NR}
Let $x_0\in\Om$.
A prime end $[E_k]\in\bdP\Om\setminus\bdySP\Om$ is said to be \emph{infinitely far away} if
there is some $x_0\in\Om$ such that whenever
$\{x_j\}_j$ is a sequence in $\Om$ with $x_j\to[E_k]$ and $\gamma_j$ is a rectifiable curve in $\Om$ connecting
a point $x_0\in\Om$ to $x_j$, $j\in\N$, we must have $\lim_{j\to\infty}\ell(\gamma_j)=\infty$. A prime end
$[E_k]\in\bdP\Om\setminus\bdySP\Om$ is said to be \emph{finitely away} if it is not infinitely far away.
\end{deff}

\begin{remk}
Recall from the beginning of Section~\ref{sect-New-Sob} that in this paper the measure $\mu$ on $X$ is doubling 
and supports a $p$-Poincar\'e inequality, and hence $X$ is quasiconvex. Given this, we know that a domain (open connected subset)
in $X$ is necessarily rectifiably connected--for each pair of points $x,y\in\Om$  there is a rectifiable curve in $\Omega$ with
$x$ and $y$ as end points. Thus the above classification of finitely away and infinitely far away prime ends 
does not depend on the choice of $x_0$.
\end{remk}

\begin{remk}\label{rem:Curve-within-End}
If a sequence $\{x_j\}_j$ of points in $\Om$
converges to the prime end $[E_k]$, then for each $k\in\N$ there exists $j_k$ such that
whenever $j\ge j_k$ we must have $x_j\in E_k$; it then follows that any curve $\gamma$ connecting $x_0$ to
$x_j$ for such $j$ must have its tail end contained in $E_k$.
\end{remk}

\begin{example}(Double equilateral comb, see Figure 2 (left).)\label{ex2}
This example also appears as~\cite[Example~5.4]{abbs}. Let $\Om\subset \R^2$ be the domain 
obtained from the unit square $(0,1)\times (0,1)$ by removing the collection of segments
$\left(0, \tfrac{3}{4}\right]\times \{\frac{1}{2^{n}}\}$
and $\left[\tfrac14, 1\right)\times \{\frac{3}{2^{n+2}}\}$
for $n=1,2, \ldots$.  Define the acceptable sets
\[
 E_{k}=\Om\cap \left(\left(\frac14-\frac{1}{2^{k+2}}, \frac34+\frac{1}{2^{k+2}}\right)
   \times \left(0, \frac{1}{2^{k}}\right)\right)
\quad \text{for } k=1,2, \ldots.
\]
Then $[E_k]$ is a prime end with impression $I[E_k]=\left[\frac14,\frac34\right]\times\{0\}$. Let $x_0\in\Om$ and without 
loss of generality assume that $x_0\in \Om\setminus E_1$ as in Figure 2 (left). If $\{x_j\}_j$ is a sequence converging 
to the prime end $[E_k]$, then a curve joining $x_0$ with any $x_j\in E_k$ has length at least 
$2k\left(\frac34-\frac14\right)\to\infty$  for $k\to \infty$, showing that $[E_k]$ is infinitely far away.
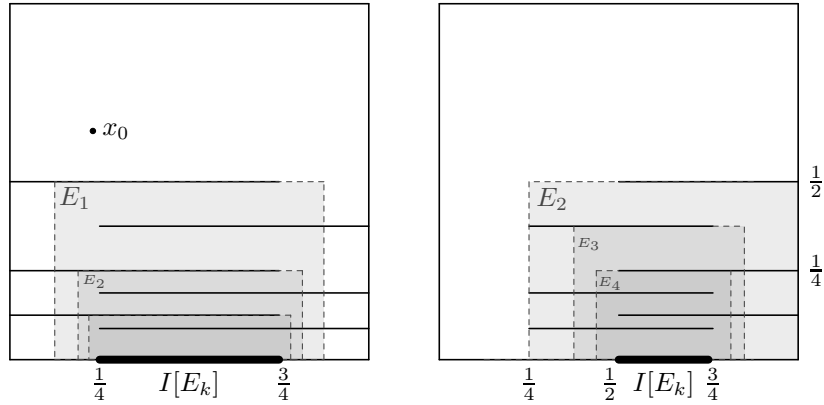
\begin{figure}[t]
\centering
\subfloat[\empty]
{
\definecolor{zzttqq}{rgb}{0.27,0.27,0.27}
\begin{tikzpicture}[line cap=round,line join=round,>=triangle 45,x=0.59cm,y=0.59cm]
\clip(-2.21,-4.32) rectangle (6.77,5.44);
\fill[color=zzttqq,fill=zzttqq,fill opacity=0.1] (-1,1) -- (-1,-3) -- (5,-3) -- (5,1) -- cycle;
\fill[line width=0.4pt,color=zzttqq,fill=zzttqq,fill opacity=0.1] (-0.48,-1) -- (-0.48,-3) -- (4.52,-3) -- (4.52,-1) -- cycle;
\fill[color=zzttqq,fill=zzttqq,fill opacity=0.1] (-0.24,-2) -- (-0.24,-3) -- (4.26,-3) -- (4.26,-2) -- cycle;
\draw [line width=0.6pt] (-2,5)-- (-2,-3);
\draw [line width=0.6pt] (-2,-3)-- (6,-3);
\draw [line width=0.6pt] (6,-3)-- (6,5);
\draw [line width=0.6pt] (6,5)-- (-2,5);
\draw [line width=0.6pt] (-2,1)-- (4,1);
\draw [line width=0.6pt] (-2,-1)-- (4,-1);
\draw [line width=0.6pt] (-2,-2)-- (4,-2);
\draw [line width=0.6pt] (0,0)-- (6,0);
\draw [line width=0.6pt] (0,-1.5)-- (6,-1.5);
\draw [line width=0.6pt] (0,-2.30)-- (6,-2.30);
\draw [dash pattern=on 2pt off 2pt,color=zzttqq] (-1,1)-- (-1,-3);
\draw [dash pattern=on 2pt off 2pt,color=zzttqq] (-1,-3)-- (5,-3);
\draw [dash pattern=on 2pt off 2pt,color=zzttqq] (5,-3)-- (5,1);
\draw [dash pattern=on 2pt off 2pt,color=zzttqq] (5,1)-- (-1,1);
\draw [line width=0.4pt,dash pattern=on 2pt off 2pt,color=zzttqq] (-0.48,-1)-- (-0.48,-3);
\draw [line width=0.4pt,dash pattern=on 2pt off 2pt,color=zzttqq] (-0.48,-3)-- (4.52,-3);
\draw [line width=0.4pt,dash pattern=on 2pt off 2pt,color=zzttqq] (4.52,-3)-- (4.52,-1);
\draw [line width=0.4pt,dash pattern=on 2pt off 2pt,color=zzttqq] (4.52,-1)-- (-0.48,-1);
\draw [dash pattern=on 2pt off 2pt,color=zzttqq] (-0.24,-2)-- (-0.24,-3);
\draw [dash pattern=on 2pt off 2pt,color=zzttqq] (-0.24,-3)-- (4.26,-3);
\draw [dash pattern=on 2pt off 2pt,color=zzttqq] (4.26,-3)-- (4.26,-2);
\draw [dash pattern=on 2pt off 2pt,color=zzttqq] (4.26,-2)-- (-0.24,-2);
\draw [fill=black] (-0.15, 2.14) circle (1pt);
\draw (-0.15, 2.14) node[anchor=west] {$x_0$};
\draw[color=zzttqq] (-0.55, 0.64) node {$E_1$};
\draw[color=zzttqq] (-0.12,-1.25) node {\tiny{$E_2$}};
\draw[color=black] (0,-3.55) node {$\frac14$};
\draw[color=black] (4.1,-3.55) node {$\frac34$};
\draw [line width=3pt,color=black] (0,-3)-- (4,-3);
\draw[color=black] (2,-3.55) node {$I[E_k]$};
\end{tikzpicture}
}
\subfloat[\empty]
{
\definecolor{zzttqq}{rgb}{0.27,0.27,0.27}
\begin{tikzpicture}[line cap=round,line join=round,>=triangle 45,x=0.59cm,y=0.59cm]
\clip(-2.21,-4.32) rectangle (6.77,5.44);
\fill[color=zzttqq,fill=zzttqq,fill opacity=0.1]  (0,1) -- (0,-3) -- (6,-3) -- (6,1) -- cycle;
\fill[line width=0.4pt,color=zzttqq,fill=zzttqq,fill opacity=0.1] (1,0) -- (1,-3) -- (4.8,-3) -- (4.8,0) -- cycle;
\fill[color=zzttqq,fill=zzttqq,fill opacity=0.1] (1.5,-1) -- (1.5,-3) -- (4.5,-3) -- (4.5,-1) -- cycle;
\draw [line width=0.6pt] (-2,5)-- (-2,-3);
\draw [line width=0.6pt] (-2,-3)-- (6,-3);
\draw [line width=0.6pt] (6,-3)-- (6,5);
\draw [line width=0.6pt] (6,5)-- (-2,5);

\draw [line width=0.6pt] (2,1)-- (6,1);
\draw [line width=0.6pt] (2,-1)-- (6,-1);
\draw [line width=0.6pt] (2,-2)-- (6,-2);
\draw [line width=0.6pt] (0,0)-- (4.1,0);
\draw [line width=0.6pt] (0,-1.5)-- (4.1,-1.5);
\draw [line width=0.6pt] (0,-2.30)-- (4.1,-2.30);

\draw [dash pattern=on 2pt off 2pt,color=zzttqq] (0,1)-- (0,-3);
\draw [dash pattern=on 2pt off 2pt,color=zzttqq] (-1,-3)-- (5,-3);
\draw [dash pattern=on 2pt off 2pt,color=zzttqq] (5,1)-- (0,1);
\draw [line width=0.4pt,dash pattern=on 2pt off 2pt,color=zzttqq] (1,0)-- (1,-3);
\draw [line width=0.4pt,dash pattern=on 2pt off 2pt,color=zzttqq] (1.5,-1)-- (2,-1);
\draw [line width=0.4pt,dash pattern=on 2pt off 2pt,color=zzttqq] (4.8,-3)-- (4.8,0);

\draw[color=black] (6.4,1) node {$\frac12$};
\draw[color=black] (6.4,-1.) node {$\frac14$};

\draw [dash pattern=on 2pt off 2pt,color=zzttqq] (1.5,-1)-- (1.5,-3);
\draw [dash pattern=on 2pt off 2pt,color=zzttqq] (4.1,0)-- (4.8,0);
\draw [dash pattern=on 2pt off 2pt,color=zzttqq] (4.5,-3)-- (4.5,-1);
\draw[color=zzttqq] (0.52, 0.62) node {$E_2$};
\draw[color=zzttqq] (1.33,-0.4) node {\tiny{$E_3$}};
\draw[color=zzttqq] (1.8,-1.28) node {\tiny{$E_4$}};
\draw[color=black] (0,-3.55) node {$\frac14$};
\draw[color=black] (1.8,-3.55) node {$\frac12$};
\draw[color=black] (4.1,-3.55) node {$\frac34$};
\draw [line width=3pt,color=black] (2,-3)-- (4,-3);
\draw[color=black] (3,-3.55) node {$I[E_k]$};
\end{tikzpicture}
}
\caption{Example~\ref{ex2} (left), Example~\ref{ex3} (right).}
\end{figure}
\end{example}

\begin{example}(see Figure 2 (right).)\label{ex3}
An example of a finitely away prime end comes from the following variant of the previous example. Let
\[
\Om=(0,1)^2\setminus \bigcup_{n\in\N}\left[\tfrac12,1\right]\times\left\{\tfrac{1}{2n}\right\}
  \setminus\bigcup_{n\in\N} \left[\tfrac14, \tfrac34\right]\times\left\{\tfrac{1}{2n+1}\right\}.
\]
This domain has only one non-singleton prime end, namely the one corresponding to the chain given by
\[
E_k=\Om\cap \left[\tfrac12-\tfrac{1}{2k},\tfrac34+\tfrac{1}{2k}\right]\times\left(0,\tfrac1k\right), \quad k\ge 2.
\]
One verifies directly from Definition~\ref{def-chain} that $\{E_k\}_{k=2}^{\infty}$ is a chain and defines an end in $\Om$.
That this end is minimal (i.e. the prime end) can be seen, by the following reasoning. Every other $E_k$, for $k\geq 3$, has 
one vertical side of its boundary which touches a slit from one family of slits, say 
$\left[\tfrac12,1\right]\times\left\{\tfrac{1}{2n}\right\}$ for $n=1,2,\dots$ and another vertical side whose distance to the 
appropriate slit from the other family ($\left[\tfrac14, \tfrac34\right]\times\left\{\tfrac{1}{2n+1}\right\}$ for $n=1,2,\dots$) 
approaches $0$ when $k\to \infty$, see Figure~2(right). Therefore, gaps between the boundaries of $E_k$ and 
corresponding teeth of the comb close for increasing $k$. Hence, if $[F_l]$ is an end dividing $[E_k]$, then 
the connectedness of all acceptable sets $F_l$ implies that every $F_l\supset E_K$ for $K$ large enough, and 
thus also $F_l\supset E_k$ for $k\geq K$ (Definition~\ref{def-chain}(a)). Since this gives the equivalence of 
$[F_l]$ and $[E_k]$, the end $[E_k]$ is, in fact, the prime end with impression $I[E_k]=[\tfrac12, \tfrac 34]$.

The prime end $[E_k]$ is finitely away. To see this, note that every point $p\in E_k$, $k\in\N$, can be connected to
the point, say, $x_0=(\tfrac18,\tfrac12)\in\Om$ by the concatenation of a horizontal line segment from 
$p=(x,y)$ to $q=(\tfrac18,y)$ with a vertical line segment from $q$ to $x_0$. Such curves have length at most $\tfrac54$ for $k>3$. 
\end{example}

 \subsection{Perron solution with respect to the prime end boundary}

 In this section we recall basic notions for the Perron method in metric measure spaces. See~\cite[Section~7]{ES} 
 for further details.

 Let $\Om\subset X$ be a domain in $X$. We say that a function $u\in \Np(\Om)$ is $p$-harmonic if it is a 
 continuous minimizer of the $p$-Dirichlet energy, i.e. for all $\phi \in \Npz(\Om)$ we have
 \[
  \int_{\phi\not=0}g_u^p\,d\mu \leq  \int_{\phi\not=0}g_{u+\phi}^p\,d\mu.
 \]
 Here, $g_u$ and $g_{u+\phi}$ stand for the minimal $p$-weak upper gradients of $u$ and $u+\phi$, respectively. 
 Recall the notion of Sobolev $p$-capacity from~\eqref{Sob-cap}.
 By modifying $u$ on a set of Sobolev $p$-capacity zero if necessary, we see that $u$ is locally H\"older 
 continuous in $\Om$, see e.g.~the discussion in~\cite[Section 5]{KiSh01} and~\cite{BBbook}. 
 
 Recall that we require $\Om$ to satisfy Assumption~\ref{assumption-ends}. 
 Moreover, we assume that $\Cp(X\setminus \Om)>0$. 
 This latter assumption allows us to avoid trivial solutions to the $p$-Dirichlet problem, see the discussion following 
 in~\cite[Definition~3.5]{ES}. Indeed, if $\Cp(X\setminus \Om)=0$, then $\Np(\Om)=\Npz(\Om)=\Np(X)$ 
 and so for any $f\in \Np(\Om)$ the constant function $u=0$ would act as a $p$-harmonic function in $\Om$ with
 $f-u\in\Npz(\Om)$, that is, the Dirichlet problem for boundary data in $\Np(X)$ has only trivial solutions.

The following notion of capacity is a modification of the related notion of $\bCp$ from~\cite[Section~3]{bbs1},
and was first formulated in~\cite[Definition~6.1]{ES}. Here we use the prime end topology as in Definition~\ref{PrimeEndTop}.

\begin{deff}[cf. Definition 6.1 in~\cite{ES}]\label{def:capP}
 Let $E\subset \clOmP$. We define \emph{prime end capacity} of set $E$ as
 \begin{equation}
  \CapPp(E, \Om)=\inf_{u\in \A_{E}}\,\|u\|^p_{\Np(\Om)},
 \end{equation}
 where $\A_{E}$ consists of all functions $u\in \Np(\Om)$ satisfying the following two conditions:
 \begin{align*}
 &(1)\,\, u\geq 1\hbox{ on } E\cap \Om,\\
 &(2)\,\liminf_{\Omega\ni y\overset{\clOmP} \to x} u(y)\geq 1\hbox{ for all } x\in E\cap \bdP \Om.
 \end{align*}
 If the underlying domain $\Om$ is fixed, we denote, for simplicity, $\CapPp(E):=\CapPp(E, \Om)$.
\end{deff}

 By~\cite[Lemma~6.2]{ES} the prime end capacity defines an outer measure on $\clOmP$. Moreover, if 
 $X$ is doubling and supports a $p$-Poincar\'e inequality, then $\CapPp$ is an outer 
 capacity, see~\cite[Proposition~6.3]{ES}. By an outer capacity we mean that for each $E\subset\clOmP$,
 \[
 \CapPp(E)=\inf_{E\subset U\subset\clOmP}\CapPp(U),
 \]
 where the infimum is over all open (in the prime end topology of $\clOmP$) subsets $U$ of $\clOmP$
 containing $E$.

 \begin{deff}[cf. Definition 6.4 in~\cite{ES}]\label{Cp-quasicont}
 Let $E\subset\overline{\Om}^P$.
  A function $f:A \to \eR$ is said to be \emph{$\CapPp$-quasicontinuous}  if for every $\eps>0$ 
  there exists an open set $U\subset \clOmP$ such that $\CapPp(U)<\eps$ and 
  $f|_{A\setminus U}$ is real-valued continuous.
 \end{deff}
 
 Recall from~\cite{Sh-rev, BBbook} that a function $f\in\Np(X)$ is $C_p$-quasicontinuous.
 
 \begin{deff}\label{def:H-Omega-f}
For $f\in N^{1,p}(\Om)$ we set $H_\Om f$ to be the function in $N^{1,p}(\Om)$ such that $f-H_\Om f\in N^{1,p}_0(\Om)$ and
whenever $\pip\in \Npz(\Om)$ (see Definition~\ref{def:zero-boundary}), we have
\[
\int_{\Om\cap\{\pip\ne 0\}}g_{H_\Om f}^p\, d\mu\le \int_{\Om\cap\{\pip\ne 0\}}g_{\pip+H_\Om f}^p\, d\mu.
\]
\end{deff}

See~\cite[Theorem~3.2]{KiMa02} for the existence and uniqueness of such a function $H_\Om f$ given $f$, 
where $H_\Om f$ is given as the solution of the obstacle problem $K_{-\infty,f}(\Om)$.

 \begin{deff}[cf. Definition 7.1 in \cite{ES}]\label{def-p-super}
  Let $1<p<\infty$. We say that a lower semicontinuous function $u:\Om\to (-\infty, \infty]$ such that 
  $u\not\equiv \infty$ on $\Om$ is \emph{$p$-superharmonic} if it satisfies the following comparison principle:
  For every nonempty open set $V\Subset \Om$ and all Lipschitz functions $v$ on $X$, it holds that 
  $H_Vv\leq u$ in $V$ whenever $v\leq u$ on $\bdy V$.
  
  A function $u$ is called \emph{$p$-subharmonic} if $-u$ is $p$-superharmonic.
 \end{deff}
 
For more information on $p$-super(sub)harmonic functions and $p$-harmonic extensions in the metric 
 setting we refer, for instance, to~\cite{bbs1} and~\cite[Chapters~9,10]{BBbook}.

 Now we are ready to describe the Perron method. 

 \begin{deff}[cf. Definition 7.2 in \cite{ES}]
  Let $f:\bdP \Om \to \eR$.  The collection of all $p$-superharmonic functions $u$ on $\Om$ 
  bounded from below such that
  \[
   \liminf_{\Omega\ni y\overset{\clOmP} \to [E_n]} u(y) \geq f([E_n]),\quad \hbox{ for all } [E_n]\in \bdP \Om
  \]
  is denoted by $\UU_f(\clOmP)$.
  We define the \emph{upper Perron solution} of $f$ by
  \[
   \uHp_{\clOmP} f(x) :=\inf_{u\in \UU_f(\clOmP)} u(x),\quad x\in \Om.
  \]
  Similarly, let $\LL_f(\clOmP)$ be the set of all $p$-subharmonic functions $u$ on $\Om$ bounded above such that
  \[
   \limsup_{\Omega\ni y\overset{\clOmP} \to [E_n]} u(y) \leq f([E_n]),\quad \hbox{ for all } [E_n]\in \bdP \Om.
  \]
  We define the \emph{lower Perron solution} of $f$ by
  \[
   \lHp_{\clOmP} f(x) :=\sup_{u\in \LL_f(\clOmP)} u(x),\quad x\in \Om.
  \]
 \end{deff}

 Note that $\lHp_{\clOmP} f=-\uHp_{\clOmP}(-f)$.

 \begin{deff}
  If $f:\bdP \Om\to\eR$ such that 
  $$
   \uHp_{\clOmP} f=\lHp_{\clOmP} f \quad \hbox{on }\Om,
  $$
  then we say that $f$ is \emph{resolutive}, and set $P_{\clOmP} f:=\uHp_{\clOmP}f$.
 \end{deff}

 One of the results obtained for the Perron method in~\cite{ES} is a comparison principle
 between $p$-super- and $p$-subharmonic functions with respect to the prime end boundary. 
 Among its consequences we have that if $f:\bdP \Om \to \R$, then $\lHp_{\clOmP}(f)\leq \uHp_{\clOmP}(f)$
 (recall that we assume $\Om$ to be bounded).

For the reader's convenience we now recall the two resolutivity results of \cite{ES} needed in our work, see 
  Section~\ref{sect:res} below for their application.

\begin{thm}[Theorem 7.7 in \cite{ES}]\label{thm77} 
  Let $F: \clOmP \to \R$ be a $\CapPp$-quasicontinuous function such that $F|_{\Om}$ is in 
  $\Np(\Om)$. Then $F$ is resolutive and $P_{\clOmP}F=H_{\Om}F$.
\end{thm}

 The following result shows the stability of the Perron solution under perturbations on a set of $\CapPp$ capacity zero.

\begin{thm}[Theorem 7.8 in \cite{ES}]\label{thm78}
  Let $f: \clOmP \to \R$ be a $\CapPp$-quasicontinuous function such that $f|_{\Om}$ is in $\Np(\Om)$. If 
  $h:\overline{\Om}^P\to \eR$ is zero in $\Om$ and is zero $\CapPp$ quasi-everywhere in $\bdP\Om$, 
  then $f+h$ is resolutive with respect to $\clOmP$, and
  $P_{\clOmP}(f+h)=P_{\clOmP} f$.
\end{thm}

We now summarize the assumptions we make throughout the paper: \emph{we assume that $X$ is complete and that the measure
$\mu$ on $X$ is doubling and supports a $p$-Poincar\'e inequality (for a fixed $1<p<\infty$). We also assume
that $\Om\subset X$ is a domain with $\mu(X\setminus\Om)>0$ such that $\Om$ satisfies Assumption~\ref{assumption-ends}.}

\section{The collection of all non-singleton prime ends is a prime end capacitary null set}

In light of Theorems~\ref{thm78}, it is desirable to know which prime ends influence the Perron solution.
Given that the Newtonian Sobolev class $N^{1,p}$ considered here is based on rectifiable curves, it would be natural to know that
the collection of all rectifiably inaccessible singleton prime ends and the collection of all non-singleton prime ends
do not play a role in determining the Perron solutions. We will prove this natural claim in this section by showing that
these two classes of prime ends have $\CapPp$-capacity zero.

\begin{lem}\label{lem:Excp:NR}
Let $\mathcal{C}$ be the collection of all prime ends in $\Om$ that are infinitely far away. Then
$\CapPp(\mathcal{C})=0$.
\end{lem}

\begin{proof}
Fix $x_0\in\Om$, and
for $\eps>0$ let $u_\eps$ be given by
\[
 u_\eps(x)=\min\{1,\eps \inf_{\gamma_x}\ell(\gamma_x)\},
\]
where the infimum is taken over all rectifiable curves in $\Om$ connecting $x_0$ to $x$. 
Since $X$ is quasiconvex, connectedness of $\Om$ 
is equivalent to its rectifiable connectedness. Note also that $u_\eps$ is finite-valued in $\Om$.

By Definition~\ref{NSP-R+NR}, if $\{x_j\}_j$ is a sequence in $\Om$ converging to $[E_k]\in\mathcal{C}$, 
then for curves $\gamma_j$ connecting $x_0$ to $x_j$, it holds that $\lim_j l(\gamma_j)=\infty$. As a consequence,
\[
\lim_{j\to \infty} u_\eps(x_j)=1,
\]
and hence, $u_\eps$ satisfies condition~(2) in Definition~\ref{def:capP}, 
and so is admissible for computing the prime end capacity of  $\mathcal{C}$. 

Note that $\lim_{\eps\to 0^+}u_\eps=0$ pointwise in $\Om$, and $0\le u_\eps\le 1$ on $\Om$. Therefore, the
Lebesgue dominated convergence theorem implies that $u_\eps\to 0 \hbox{ in } L^p(\Om)$.

We now show that $u_\eps\in \Np(\Om)$ by proving that the function $\rho_\eps=\eps$ is an upper gradient 
of $u_\eps$. Let $x,y\in \Om$ and $\gamma_{xy}$ be a curve joining $x$ and $y$ in $\Om$, 
while $\gamma_{yx_0}$ be a curve joining $y$ and $x_0$ in $\Om$. Furthermore, without loss of generality, 
suppose that $u_{\eps}(x)> u_{\eps}(y)$. Then, by the definition of $u_{\eps}$,
\[
 u_{\eps}(x)\leq \eps \ell(\gamma_{yx_0}+\gamma_{xy})=\int \limits_{\gamma_{yx_0}\,+\,\gamma_{xy}}\eps\,ds
   =\int \limits_{\gamma_{yx_0}}\eps\,ds+\int \limits_{\gamma_{xy}} \eps\,ds,
\]
as curve $\gamma_{yx_0}+\gamma_{xy}$ is admissible for $u_{\eps}$. 
As $u_\eps(y)<u_\eps(x)\le 1$, it follows that $u_\eps(y)\ne 1$. Therefore,
by taking the infimum 
over all rectifiable curves $\gamma_{yx_0}$ we arrive at
\[
 u_{\eps}(x)\leq \int \limits_{\gamma_{xy}} \eps\,ds + u_{\eps}(y).
\]
As we assumed that $u_\eps(x)> u_\eps(y)$, it follows that
\[
|u_\eps(x)-u_\eps(y)|=u_\eps(x)-u_\eps(y)\le \int_{\gamma_{xy}}\eps \, ds.
\]
Thus, $\rho_\eps=\eps$ satisfies the definition of an upper gradient of $u_{\eps}$. 
Lemma~A.2 in~\cite{abbs} allows us to infer that $u_{\eps}$ is measurable in $\Om$. 
The conclusion of the lemma now follows from the fact that 
\[
\lim_{\eps\to 0^+}\Vert u_\eps\Vert_{N^{1,p}(\Om)}=0.
\]
\end{proof}

In the next lemma we associate with every finitely away prime end in $\clOmP$ and with 
every inaccessible singleton prime end in $\clOmP$ a curve $\beta$, such that its substantial part is contained 
in $\partial \Om$. This result is then used to show Proposition~\ref{prop:exceptional}, which deals with a claim 
analogous to Lemma~\ref{lem:Excp:NR} for this class of prime ends.

\begin{lem}\label{lem:Positivity-At-Bdry}
Let $[E_k]$ be a finitely away prime end or a rectifiably inaccessible prime end,
and $\gamma_j$ be a sequence of rectifiable curves in $\Om$ with
end points $x_0,x_j\in\Om$ such that $x_j\to[E_k]$ as $j\to\infty$ and
$\sup_j\ell(\gamma_j)<\infty$. Suppose that $\beta:[0,L]\to \overline{\Om}$ is an
arc-length parametrized rectifiable curve such that $\gamma_j\to\beta$ uniformly as $j\to \infty$. Then 
$\mathcal{H}^1(|\beta|\cap\partial\Om)>0$ and for each $\eps>0$ there exists $j_\eps\in\N$ such that for $j>j_\eps$,
at least $\mathcal{H}^1(|\beta|\cap\partial\Om)/2$ length of the curve $\gamma_j$ is within a distance $\eps$ of
$X\setminus\Om$. 
\end{lem}

Since $X$ is proper and each $\gamma_j$ is (under its arc-length parametrization) $1$-Lipschitz, an application
of the Arzela-Ascoli theorem tells us that after passing to a subsequence if necessary, we do always have 
a curve $\beta$ in $\overline{\Om}$ such that $\gamma_j$ converges uniformly to $\beta$.

\begin{proof}
Note that $x_0$ is one end point of $\beta$. Let $x_\infty$ be the other end point of $\beta$. Then
as $x_j\to[E_k]$, it follows that
$x_\infty\in I[E_k]$.  Let $L:=\sup_j\ell(\gamma_j)$. We may assume that each $\gamma_j$ is arc-length
parametrized and then extended by constant to $[\ell(\gamma_j),L]$, so $\gamma_j:[0,L]\to \Omega$
is $1$-Lipschitz. We can then represent $\beta$ by the parametrization $\beta:[0,L]\to\overline{\Omega}$
given by $\beta(t)=\lim_{j\to\infty}\gamma_j(t)$.

Suppose that
$\mathcal{H}^1(|\beta|\cap\partial\Om)=0$. Then we can find a sequence $t_m\in[0,L]$, with
$\beta(t_m)\to I([E_k])$ and $t_m$ increasing
strictly monotonically as $m\to \infty$, such that $\beta(t_m)\in\Om$. We can also choose
$t_1$ so that
$\beta([0,t_1])\subset\Om$. We can furthermore ensure by Remark~\ref{rem:Curve-within-End} that for 
each $m\in\N$ the points $\beta(t_m)$ lie in $E_{m+1}$ and $\beta\vert_{[t_m,t_{m+1}]}\subset\overline{E}_{m+1}$.
The latter is thanks to knowing that $\ell(\beta)<\infty$ and there is a fixed positive $d_M$-distance 
between $\Om\cap\partial E_m$ and $\Om\cap\partial E_{m+1}$.
Let
\[
r_m=(4C)^{-m-1}\, \min\big\{\text{dist}(\beta(t_m),\partial E_m),  \text{dist}(\beta(t_{m+1}),\partial E_{m+1}), 1\big\}.
\]
Here $C$ is the quasiconvexity constant of $X$.
Since $\gamma_j\to\beta$ uniformly, there is a sufficiently large $j$ 
such that the segment $\beta_m:=\gamma_j\vert_{[t_m,t_{m+1}]}$ of $\gamma_j$ 
satisfies
\[
\ell(\beta_m)\le [1+2^{-m}]|t_{m+1}-t_m|,
\] 
and
\[
d(\beta(t_m),\beta_m(t_m))<r_m,\qquad d(\beta(t_{m+1}),\beta_m(t_{m+1}))<r_m.
\]
By choosing $j$ to be large enough, we can also ensure that $\beta_m$ is contained in $E_{m-1}$.
Next, by the quasiconvexity of $X$, we can find $C$-quasiconvex curves
$\alpha_m,\widehat{\alpha}_m$ with end points $\beta(t_m),\beta_m(t_m)$ and end points 
$\beta(t_{m+1}), \beta_m(t_{m+1})$, respectively. Then 
\[
\ell(\alpha_m)+\ell(\widehat{\alpha}_m)\le 4^{-m},
\]
with $\alpha_m\subset E_m$ and $\widehat{\alpha}_m\subset E_{m+1}$.

Now the concatenated curve $\Gamma:[0, \ell(\Gamma)]\to X$ given by
\[
\Gamma=\beta\vert_{[0,t_1]}+\sum_m [\alpha_m+\beta_m+\widehat{\alpha}_m]
\]
is a rectifiable curve with
\[
  \ell(\Gamma)\le \ell(\beta)+\sum_m\bigg(4^{-m}+\ell(\beta_m\vert_{[t_m,t_{m+1}]})\bigg)
    \le 3L+1<\infty.
\]
Moreover, $\Gamma([0,\ell(\Gamma))\subset\Om$, $\Gamma(\ell(\Gamma))=x_\infty$, and for each $m$ a tail end of $\Gamma$ 
lies in $E_m$. More precisely, there is some $s_m$ such that
$\Gamma((s_m,\ell(\Gamma)))\subset E_m$. This makes $x_\infty$ accessible through the end $[E_k]$
(see also Remark~\ref{rem:SP}), which makes
$[E_k]$ an accessible prime end, violating the assumption that it is either a non-singleton prime end (that is 
finitely away) or an inaccessible prime end,
and so it must be that $\mathcal{H}^1(|\beta|\cap\partial\Om)>0$.

 Note that in the above proof, we only needed a sequence $t_m$, $m\in\N$, with $t_m\to L$ as $m\to\infty$, such that
$\beta(t_m)\in\Om\cap E_m$ in order to gain a contradiction. Thus the above proof gives us a stronger conclusion; 
namely, there is some $s\in (0, \ell(\beta))$ such that $\beta\vert_{[s,\ell(\beta)]}\subset\partial\Om$. Now the uniform convergence of 
$\gamma_j$ to $\beta$ yields the verity of the final claim of the lemma.
\end{proof}

The following is the first main result of this paper.

\begin{thm}\label{prop:exceptional}
Let $\mathcal{F}$ be the collection of all  non-singleton prime ends of $\Om$ together with all
singleton prime ends that are rectifiably inaccessible. Then
$\CapPp(\mathcal{F})=0$.
\end{thm}

\begin{proof}
By Lemma~\ref{lem:Excp:NR} we know that the collection of all prime ends that are infinitely far away has
$\CapPp$-capacity zero. Thus it suffices to focus on the subcollection of all finitely away prime ends together with
rectifiably inaccessible singleton prime ends.

For each $\eta>0$ we set $\Om_\eta:=\{x\in\Om\, :\, \text{dist}(x,X\setminus\Om)<\eta\}$. Then as $\Om$ is
bounded, $\lim_{\eta\to 0^+}\mu(\Om_\eta)=0$. Thus, for each positive integer $n$ we can find $\eta_n>0$ such that
$\mu(\Om_{\eta_n})<2^{-np}$. From now on let us fix such a sequence $(\eta_n)$.

Fix $\eps>0$ and for each positive integer $n$ set $\rho_n:=n\, \chi_{\Om_{\eta_n}}$. With this choice of $\rho_n$, let $w_n$ be a function on $\Om$ given by
\[
 w_n(x)=\min\bigg\lbrace 1,\inf_\gamma\int_\gamma\bigg(\frac{1}{n}+\sum_{j=n}^\infty\rho_j\bigg)\, ds\bigg\rbrace,
\]
where the infimum is over all rectifiable curves $\gamma$ connecting $x$ to $x_0$ in $\Om$.  Note that by the same 
reasoning as in the proof of Lemma~\ref{lem:Excp:NR}, we have that $g_n:=\frac{1}{n}+\sum_{j=n}^\infty\rho_j$ is an upper 
gradient of $w_n$, and that
\begin{equation}\label{upp-grad:prop}
\left(\int_\Om g_n^p\, d\mu\right)^{1/p}
\le \frac{\mu(\Om)^{\frac1p}}{n}+\sum_{j=n}^\infty j \mu(\Om_{\eta_j})^{1/p}\le \frac{\mu(\Om)^{\frac1p}}{n}+\sum_{j=n}^\infty 2^{-j}\, j\to 0,
\ \text{ as }n\to\infty.
\end{equation}

We claim that $w_n$ is admissible for computing $\CapPp(\mathcal{F})$.
To see this, suppose this is not the case. Then there is some $[E_m]\in \mathcal{F}$, $\eps>0$, and a
sequence $x_l\in\Om$ with $x_l\to [E_m]$ but $\sup_l w_n(x_l)\le 1-\eps$. 
Then for each $j$ there is a curve
$\gamma_l$ with end points $x_0$ and $x_l$ such that 
\[
\ell(\gamma_l)/n\le 1-\eps\text{ and }\int_{\gamma_l}\sum_{j=n}^\infty\rho_j\, ds<1-\eps. 
\]
In this
case we have that $\ell(\gamma_l)\le n(1-\eps)$, and hence, by passing to a further subsequence if necessary
(and invoking the Arzela-Ascoli theorem),
we have $\gamma_{l_k}\to\beta$ uniformly for some rectifiable curve $\beta$ in $\overline{\Om}$, and hence
by Lemma~\ref{lem:Positivity-At-Bdry}, there is some $L>0$ and some $m_0\in\N$ such that 
whenever $k\ge m_0$ we have $\mathcal{H}^1(|\gamma_{l_k}|\cap\Om_{\eta_{n_L}})\ge n_LL$ where
$n_L$ is the smallest positive integer that is not smaller than $\max\{1/L,n+1\}$. It then follows that, with the choice of 
$k_0=\max\{n+1,m_0+1\}$,
\[
1-\eps\ge \int_{\gamma_{l_{k_0}}}\sum_{j=n}^\infty\rho_j\, ds\ge \int_{\gamma_{l_{k_0}}}\rho_{n_L}\, ds\ge 1,
\]
which is not possible.
It follows that  $w_n$ is admissible.
The definition of $w_n$ together with~\eqref{upp-grad:prop} allow us to conclude that
$w_n\in \Np(\Om)$ for each $n\in\N$ and that $\lim_n\Vert w_n\Vert_{\Np(\Om)}=0$.
In addition from the above argument, we have $w_n\in\mathcal{A}_{\mathcal{F}}$ for each $n\in\N$.
Indeed, to see this it only remains to show that $\lim_n\int_\Om w_n^p\, d\mu=0$. This would follow from the
Lebesgue dominated convergence theorem if we know that $w_n\to 0$ pointwise in $\Om$. For each $n\in\N$
set $U_n$ to be the collection of all points $x\in \Om$ for which there is a curve $\gamma$ connecting $x$ to $x_0$
with $\gamma\subset\Om\setminus\Om_{\eta_n}$ and $\ell(\gamma)<\sqrt{n}$. It is not difficult to see that
$\bigcup_n U_n=\Om$, $U_n\subset U_{n+1}$,
and $w_n\le 1/\sqrt{n}$ on $U_n$  with $w_n\le 1$ on $\Om\setminus U_n$.
It follows that $w_n\to 0$ pointwise in $\Om$. 
Hence, $\CapPp(\mathcal{F})=0$ and the proof of the proposition is completed.
\end{proof}

\begin{remk}
 In \cite[Remark 8.4(b)]{ES} the following question was posed: \emph{Are there bounded domains for which
 \[
  \CapPp(\bdP\Om \setminus \bdySP \Om)>0\,?
 \]
 }
 The results of this section give the negative answer to this question. Indeed, by 
 Theorem~\ref{prop:exceptional} we have that the collection of all non-singleton prime ends 
 $\bdP\Om\setminus\bdySP\Om$ has $\CapPp$-capacity zero.
\end{remk}

\section{Resolutivity of functions that are $d_M$-Lipschitz outside $\mathcal{F}$}\label{sect:res}

The following is the second main result of the paper. 

\begin{thm}\label{thm:main}
Let $(X,d,\mu)$ be a complete metric  measure space equipped with a doubling measure $\mu$ supporting a
$p$-Poincar\'e inequality and $\Om\subset X$ be a bounded domain such that 
$\Cp(X\setminus\Om)>0$. 
Let $f:\bdP\Om\to\R$ such that its restriction to $\partial_{RSP}\Om$, the collection of all 
rectifiably accessible prime ends of $\Om$, is Lipschitz continuous with respect to the Mazurkiewicz metric $d_M$. Then
$f$ is resolutive and $P_{\overline{\Om}^P}f=H_\Om f$.

Furthermore, if $f_1, f_2$ are such functions with $f_1=f_2$ on $\bdyRSP\Om$, then
$P_{\overline{\Om}^P}f_1=P_{\overline{\Om}^P}f_2$ on $\Om$.
\end{thm}

Since the closure of $\bdyRSP\Om$ under the Mazurkiewicz metric is $\bdySP\Om$, we know that if 
$f_1=f_2$ on $\bdyRSP\Om$, then the continuous extensions of $f_1$ and $f_2$ to $\bdySP\Om$ also satisfy this 
equality.

\begin{proof} 
Recall from the discussion following Definition~\ref{Mazur} that the Mazurkiewicz metric
$d_M$ extends to a metric, also denoted $d_M$, to
$\Om\cup \bdySP\Om$. Therefore by the McShane extension theorem, see~\cite{mcs} 
and~\cite[Chapter~6]{heinonen}, $f\vert_{\partial_{RSP}\Om}$ has a Lipschitz extension, denoted 
$F$, to $\Om\cup\bdySP\Om$. Moreover, the Lipschitz constants of $f\vert_{\partial_{RSP}\Om}$ and $F$ are the same.

From~\cite[Lemma~6.2.6]{HKST} it follows that
\[
\Lip_M F(x)=\limsup_{y\to x}\frac{|F(y)-F(x)|}{d_M(y,x)}
\]
is a bounded upper gradient of $F$ in $\Om$ with respect to the metric $d_M$, and hence by~\cite[Proposition~5.3]{bbs1}
we know that $F\vert_\Om\in N^{1,p}(\Om)$.

We extend $F$ by $f$ to $\bdP\Om\setminus\bdySP\Om$. Observe that by the discussion following 
Definition~\ref{SP-R+NR} and by Definition~\ref{NSP-R+NR}, every prime end in $\partial_{P}\Om$ falls into one 
of the classes described by these definitions. Therefore, 
Theorem~\ref{prop:exceptional} implies 
that $F$ is $\CapPp$-quasicontinuous on $\Om\cup\bdP\Om$. 
Hence, by Theorem~\ref{thm77} we know that $F\vert_{\bdP\Om}=f$ is resolutive, and that
$P_{\overline{\Om}^P}f=HF$. This proves the first assertion of the theorem.

Let $f_1$ and $f_2$ be as in 
assumptions of the theorem and define $h:=f_2-f_1$. Then $h$ equals zero $\CapPp$ quasi-everywhere 
on $\bdP\Om$, again by Lemma~\ref{lem:Excp:NR} and Proposition~\ref{prop:exceptional}. Applying 
Theorem~\ref{thm78} with $f:=f_1$ and $h=[f_2-f_1]\chi_{\bdP\Om}$ we get the desired conclusion and the 
proof of the theorem is completed.

\end{proof}

\begin{remk}
Let us compare Theorem~\ref{thm:main} to results in~\cite{ES}. In 
\cite[Theorem~7.7]{ES} (see Theorem~\ref{thm77} above), a corresponding 
function $f$ is defined on the whole $\clOmP$, is $\CapPp$-quasicontinuous and is assumed to 
belong to $\Np(\Om)$. Here, we require $f$ to be merely defined on $\bdP \Om$, but then also to be 
Lipschitz continuous on $\partial_{RSP}\Om$ with respect to the Mazurkiewicz metric $d_M$. See~\cite{bbs1} for
more on this metric and its relation to $\bdySP\Om$.  Thus the advantage in the above theorem is 
that we do not a priori have to verify whether the function is a quasicontinuous function on $\clOmP$,
but the disadvantage is that we need $f$ to be Lipschitz on $\partial_{RSP}\Om$ with respect to 
$d_M$, not merely $\CapPp$-quasicontinuous. The examples in the next section illustrate the strength
and limitations of these results.
\end{remk}

In light of the above remark, the following proposition is a strengthening of Theorem~\ref{thm:main}.
Its proof is similar to the latter part of the proof of Theorem~\ref{thm:main}, and hence is omitted here.

\begin{prop}\label{prop:ext}
In the setting of Theorem~\ref{thm:main}, if $F:\Om\cup\partial_{RSP}\Om\to\R$ is $\CapPp$-quasicontinuous
and $F\vert_{\Om}\in \Np(\Om)$, then $F$ is resolutive.
\end{prop}

The following observation is a consequence of Proposition~\ref{prop:ext}. The proof is the same as 
that of~\cite[Corollary 7.9]{ES} and, thus is omitted. 
\begin{cor}\label{cor:ext}
 In the setting of Theorem~\ref{thm:main}, let $F:\Om\cup\partial_{RSP}\Om\to\R$ be a bounded
 $\CapPp(\Om)$-quasicontinuous
and $F\vert_{\Om}\in \Np(\Om)$. Moreover, let $u$ be a bounded $p$-harmonic function on $\Om$. If $E\subset \bdP \Om$ is 
such that $\CapPp(E)=0$ and, for all $x\in \bdP \Om\setminus [E\cup\mathcal{F}]$,
\[
    \lim_{\Omega\ni y\overset{\clOmP} \to x} u(y) =f(x),
\]
then $u=P_{\clOmP}f$.
\end{cor}

\section{Examples}

Some examples related to the Dirichlet problem for the prime end boundary can be found in~\cite{bbs1,ES, BjComb}.
The examples we provide here are geared more towards illustrating the properties of results from the previous sections.

\begin{example}\label{ex:Cp-cts-non-res}
In Example~\ref{ex2}, consider the function $f:\partial\Om\to\R$ given by
$f(x,0)=f(x,1)=0=f(0,y)=f(1,y)$ for $x, y\in[0,1]$, and for $n\in\N$:
\[
f(x,y)=\begin{cases} 4x&\text{ when }0\le x\le \frac14,\ \ \  y=\frac{1}{2^n},\\
                                 1&\text{ when }\frac14\le x\le \frac34, \ \ \ y=\frac{1}{2^n},\\
                                  -1&\text{ when }\frac14\le x\le \frac34, \ \ \  y=\frac{3}{2^{n+2}},\\
                                  4(x-1)&\text{ when }\frac34\le x\le 1,  \ \ \ y=\frac{3}{2^{n+2}}.\end{cases}
\]
 There is a natural pull-back $f_0$ of $f$ to $\bdP\Om$ by setting $f_0([E_k])=f((x,y))$ if $I([E_k])=\{(x,y)\}$, 
and by setting $f_0([F_k])=0$ where $[F_k]$ is the prime end with the non-singleton impression.
It is easy to see that $f_0$ is $\CapPp$-quasicontinuous on $\bdP\Om$ as it is continuous on $\bdySP\Om$, 
but neither Theorem~\ref{thm:main} nor~\cite[Theorem~7.7]{ES} tells us that $f_0$ is resolutive. Note that $f_0$ is not
continuous on $\bdP\Om$ as it fails to have a continuous extension to $\bdP\Om\setminus\bdySP\Om$. On the other hand, if 
$f_1(x,y)=\sqrt{y}f(x,y)$ for $(x,y)\in\partial\Om$ and $f_0$ is constructed in a corresponding manner from $f_1$, then it is clear
that such $f_0$ is continuous on $\bdP\Om$, and again $f_0$ is not known to be resolutive. Both constructions of $f_0$
do not give functions that are Lipschitz (with respect to $d_M$) on $\bdySP\Om$.
\end{example}

\begin{example}
We modify Example~\ref{ex2} as follows. Let $\alpha>0$ and 
\[
X=\{(x,y,z)\in\R^3\, :\, 0\le x\le 2,\, 0\le y\le 2,\, |z|\le y^\alpha\},
\]
equipped with the $3$-dimensional Lebesgue measure and the Euclidean metric.
It can be seen that the measure on this space is doubling and supports a $1$-Poincar\'e inequality.
Let $\Om$ be obtained by removing
\[
\bigcup_{n\in\N}\left([0, \tfrac34]\times\{2^{-n}\}\cup[\tfrac14,1]\times\{3\cdot 2^{-n-2}\}\right)\times[-1,1]
\]
from
\[
\{(x,y,z)\in\R^3\, :\, x,y\in(0,1),\, |z|<y^\alpha\}.
\]
Suppose that $\alpha>p-1$. Then
an extension $F$ of the function given in Example~\ref{ex:Cp-cts-non-res}, by $F(x,y,z)=f(x,y)$, yields a function
on $\bdP\Om$ that is easily
seen to have an extension to $\Om$ such that this extension is $\CapPp$-quasicontinuous on
$\clOmP$ and the restriction of the extension belongs to $\Np(\Om)$. Thus Proposition~\ref{prop:ext} now
tells us that $F$ is resolutive, even though it is not continuous on the boundary of $\Om$. 
\end{example}

\begin{example}
 Perhaps one of the important applications of Theorem~\ref{prop:exceptional} is 
 that we obtain a handy, geometric, way of verifying which prime ends form a set of $\CapPp$-capacity zero. 
 For instance, 
 consider the Euclidean planar domain $\Om$ in Example~\ref{ex3}. Since the prime end $[E_k]$, with impression 
 $I[E_k]=[\tfrac12,\tfrac34]\times \{0\}$, is finitely away, it holds that $\CapPp(\{[E_k]\})=0$. Moreover, 
 there is no prime end associated with points in $[(\tfrac14, \tfrac12)\cup (\tfrac34, 1)]\times\{0\}$. (However, there is 
 an end $[F_n]$ with impression $[\tfrac14, \tfrac12]\times \{0\}$ given e.g. by a chain with acceptable sets 
 $F_n:=\Om\cap [\tfrac14-\tfrac{1}{2n}, \tfrac12+\tfrac{1}{2n}]\times (0,\tfrac1n)$ for $n=3,4,\ldots$.) The remaining 
 prime end boundary consists of singleton prime ends only (in fact all of them are rectifiably accessible). 
 Therefore, Theorem~\ref{thm:main} allows us to conclude that any function on $\bdP \Om$ Lipschitz 
 continuous on $\bdySP \Om$ with respect to the Mazurkiewicz distance $\dM$ is resolutive. 
 Furthermore, 
 one can perturb the boundary data at a point $[E_k]\in \bdP \Om\setminus \bdySP \Om$ and that data is 
 resolutive as well. 
 Note that $C_p(I([E_k]))>0$, and so the now-classical theory of Perron solutions from~\cite{BBS2} does not
 yield such perturbation result.
\end{example}

\begin{example}
 The domain in Example~\ref{ex1} above has exactly one rectifiably inaccessible (singleton) prime end $[E_k]$ with 
 impression $\{(0,0)\}$, and thus, Proposition~\ref{prop:exceptional} gives us that $\CapPp(\{[E_k]\})=0$. The remaining 
 prime ends are rectifiably accessible. 
 Therefore, a boundary 
 data $f:\bdP \Om\to \eR$ can be perturbed at $[E_k]$ freely, remaining resolutive, provided 
 Lipschitz continuity of $f|_{\partial_{RSP} \Om}$ with respect to the Mazurkiewicz distance. Observe that for 
 $p>2$ it holds that the Sobolev capacity $C_p(\{(0,0)\})>0$, and so the resolutivity cannot be inferred 
 from~\cite{BBS2}. However, see \cite[Example 10.1]{bbs1} for a similar discussion in the context of the 
 Perron method with respect to the Mazurkiewicz boundary.
\end{example}

The following is an example of a domain with infinitely many rectifiably inaccessible prime ends.

\begin{figure}[ht]
\centering
\subfloat[\empty]
{\definecolor{ududff}{rgb}{0.30196078431372547,0.30196078431372547,1.}
\begin{tikzpicture}[line cap=round,line join=round,>=triangle 45,x=1.4cm,y=1.4cm]
\clip(-3.6513150905105403,0.433485892604058) rectangle (5.257890478323242,6.668427395244932);
\draw [line width=1.1pt] (-3.,6.)-- (-3.0016480173258384,1.1896313102665907);
\draw [line width=1.1pt] (-3.0016480173258384,1.1896313102665907)-- (4.201634413358811,1.189631310266591);
\draw [line width=1.1pt] (4.201634413358811,1.189631310266591)-- (4.193456397618059,5.999213851892812);
\draw [line width=1.1pt] (4.193456397618059,5.999213851892812)-- (-3.,6.);
\draw [line width=0.7pt] (0.5927197326301196,3.5972490098144467)-- (-2.0108070858161837,3.5864459939702713)-- (-2.0000040699720083,2.1280388550065763)-- (-0.16349137646216771,2.1280388550065763)-- (-0.16349137646216771,1.5554790152652735)-- (-1.362626135165652,1.5554790152652735)-- (-1.362626135165652,1.3286156825375877)-- (-0.4983848676316092,1.3178126666934122);
\draw [line width=0.7pt] (0.6,3.4)-- (-1.8,3.4)-- (-1.8,2.2)-- (-0.10769163111706392,2.194119866911877)-- (-0.10769163111706392,1.526800814207951)-- (-1.3406851126031294,1.526800814207951)-- (-1.336265781128269,1.3500275552135335)-- (-0.4700768120556208,1.336769560788952)-- (-0.46565748058076034,1.2439635998168828);
\draw [line width=0.7pt,dash pattern=on 4pt off 4pt] (-0.5000270850698577,1.3178331944113904)-- (-0.47818706972230657,1.3167404257994855);
\draw [line width=0.7pt,dash pattern=on 4pt off 4pt] (-0.47818706972230657,1.3167404257994855)-- (-0.47471907984648076,1.2478141270174476);
\draw (-0.9620270688895503,4.1293789279044315) node[anchor=north west] {$\tfrac12$};
\draw (-0.970510243176005,2.7072708502360055) node[anchor=north west] {$\tfrac13$};
\draw (-0.9620270688895503,1.9562168114017923) node[anchor=north west] {\footnotesize{$\tfrac14$}};
\draw (-2.28413514655797,3.1530134685055647) node[anchor=north west] {$\tfrac14$};
\draw (-0.4511167626236418,2.0063126331139933) node[anchor=north west] {\footnotesize{$\tfrac18$}};
\draw (-0.541356316904144,1.1245878422943718) node[anchor=north west] {$\bf{x_1}$};
\draw [line width=0.7pt] (2.410191160618336,2.3890088677941614)-- (1.196689328880752,2.3890088677941614)-- (1.200211193983707,1.8012087099589449)-- (2.0657934270293863,1.789193152885558)-- (2.0657934270293863,1.5146011774760892)-- (1.5335074439279544,1.5146011774760892)-- (1.5335074439279544,1.4005398953829253)-- (1.9728546045831046,1.3920909115241724)-- (1.9728546045831046,1.31182556486602)-- (1.626446266374236,1.3160500567953966)-- (1.6222217744448597,1.25268267785475);
\draw [line width=0.7pt] (2.401510557664984,2.3186961467295695)-- (1.2411269343351687,2.3154906671071114)-- (1.2443324139576268,1.8314632441159462)-- (2.0852399340356644,1.8271859054994581)-- (2.0852399340356644,1.4986083287211336)-- (1.5475675356711316,1.4964747080927028)-- (1.5475675356711316,1.4132635035839063)-- (1.982471345339084,1.4100180552771344)-- (1.9838890254640793,1.2951859651525284)-- (1.6408104352152555,1.2994390055275138);
\draw (-2.459254923698221,4.084307061620281) node[anchor=north west] {$\bf{S_1}$};
\draw (0.6359194487606271,2.7922468948079553) node[anchor=north west] {$\bf{S_2}$};
\draw [line width=0.7pt] (3.2009019973152104,1.7956050626881706)-- (2.592224680450071,1.8007633450344853)-- (2.5973829627963854,1.599590333528211)-- (2.9945707034626206,1.6047486158745257)-- (2.999728985808934,1.5193662993867865)-- (2.6654369138008245,1.5193662993867865)-- (2.6654369138008245,1.4538188342871574)-- (2.9243494009443607,1.4538188342871574)-- (2.927626774199342,1.3817166226775655);
\draw [line width=0.7pt] (2.927626774199342,1.3817166226775655)-- (2.7105303088158017,1.381315621315147)-- (2.7116760039536607,1.3469463250527896)-- (2.895022239313069,1.3484615831962556)-- (2.895507217183601,1.322762519289198)-- (2.732722934060181,1.3217322390162651);
\draw [line width=0.7pt] (3.2030033111573877,1.7592125249634927)-- (2.617384447389464,1.7608302566313598)-- (2.6190021790573312,1.6184698698590476)-- (3.0088755110133247,1.6200876015269148)-- (3.013728706016926,1.506846384776212)-- (2.6728274940976866,1.507704602044648)-- (2.6728274940976866,1.4597648692226666)-- (2.9305035580158383,1.4597648692226666)-- (2.9344985357510036,1.3718753590490336)-- (2.7227647157872505,1.3708766146152425);
\draw (1.697213041875339,1.1095638868663216) node[anchor=north west] {$\bf{x_2}$};
\draw (2.778937832694955,1.1245878422943718) node[anchor=north west] {$\bf{x_3}$};
\draw (2.323530590418101,2.231192855973742) node[anchor=north west] {$\bf{S_3}$};
\draw [line width=0.4pt] (1.6408104352152555,1.2994390055275138)-- (1.6379530222979317,1.258650940769234);
\draw [line width=0.4pt] (1.6379530222979317,1.258650940769234)-- (1.8938808242019822,1.257466089834493);
\draw [line width=0.4pt] (1.6222217744448597,1.25268267785475)-- (1.6225402217400176,1.2434167313326392);
\draw [line width=0.4pt] (1.6225402217400176,1.2434167313326392)-- (1.8797783619559447,1.242954072807071);
\draw (1.6671651310192386,2.9124385382323577) node[anchor=north west] {$\frac14$};
\draw (1.6671651310192386,2.2065521873944966) node[anchor=north west] {\footnotesize{$\frac16$}};
\draw (1.6821890864472888,1.8058335245529874) node[anchor=north west] {\tiny{$\tfrac18$}};
\draw (3.878937832694955,1.1245878422943718) node[anchor=north west] {$(\frac32, 0)$};
\end{tikzpicture}
}
\caption{Example~\ref{ex44}}
\end{figure}

\begin{example}\label{ex44} 
 Consider $\Om=(0,\frac32)\times (0, 1)$ with the tunnel $S_1$ 
 removed from $\Om$, see $S_1$ in Figure 3. The height and width 
 of the double-slit tunnel are both equal to $\tfrac12$. Following the idea of Example~\ref{ex1}, the 
 $n$-th horizontal sides of $S_1$ to be of length $\frac1n$ for $n=4, \ldots$, while the length of the 
 $n$-th vertical sides is $\frac{1}{2^n}$ for $n\ge 2$. The separation between the two sinuous curves that form $S_1$
 narrows as $n$ increases (see 
 Figure~3). The limiting point, denoted $x_1$, forms the impression 
 of three singleton prime ends. Since any curve approaching $x_1$ from inside the tunnel $S_1$ is of 
 length at least $\sum_{n=4}^{\infty}\frac{1}{n}=\infty$, this prime end is rectifiably inaccessible. 
 The other two singleton prime ends with the same impression  $\{x_1\}$, namely defined by two curves 
 approaching $x_1$ from the left- and right-hand sides of $S_1$, respectively. Existence of such prime ends is 
 guaranteed e.g. by Lemma 7.7. in \cite{abbs}.  Next, we scale the dimensions of $S_1$ by $\frac12$ and 
 shift it to the right in a distance $\frac18$ obtaining new tunnel, denoted $S_2$, and the corresponding 
 impression $\{x_2\}$ gives a rectifiably inaccessible prime end. We repeat this procedure, obtaining a family 
 $S_n$ for $n\in\N$ with $2^{-(n+2)}$ the distance between $S_n$ and $S_{n+1}$. 
 Moreover, associated with $\{S_n\}$ is a sequence of points $\{x_n\}_{n=1}^{\infty}$ such that $x_n\to (\tfrac32,0)$ 
 for $n\to \infty$. Each of $x_n$ is an impression of a rectifiably inaccessible prime end. By 
 Proposition~\ref{prop:exceptional} 
 the collection of rectifiably inaccessible prime ends of the domain
 $\Om\setminus \bigcup_{n=1}^{\infty} S_n$, has $\CapPp$-capacity zero. 
\end{example}

\section{The Kellogg property}

The aim in this section is to prove a variant of the Kellogg property for the Perron solutions $P_{\overline{\Om}^P}f$.
Note that every continuous function on $\partial\Om$ is resolutive for the classical Perron solution as 
considered in~\cite{BBS2}, and the corresponding Kellogg property is proved there. 
Since we do not know that every continuous function on $\bdP\Om$ is resolutive, in this paper we only consider
continuous data $f$ on $\bdP\Om$ that are resolutive, see
Definition~\ref{def:Res-Kellogg} below. See~\cite{Bj17} for similar restrictions for domains whose
prime end boundary consists only of singleton prime ends. Since the method of~\cite{Bj17} rests crucially on
the compactness of $\bdySP\Om$, and such a compactness property is unavailable in our, more general, setting,
our proof of the Kellogg property is different than the one in~\cite{Bj17}, and is more in the spirit of~\cite{BBS}.

\begin{deff} 
A function $u\in N^{1,p}(W)$ for some non-empty open set $W\subset X$ is a \emph{$p$-superminimizer}
in $W$ if whenever $0\le \pip\in N^{1,p}_0(W)$, we have
\[
\int_{W\cap\{\pip\ne0\}}g_u^p\, d\mu\le \int_{W\cap\{\pip\ne0\}}g_{u+\pip}^p\, d\mu.
\]
\end{deff}

Here $N^{1,p}_0(W)$ consists of functions $f\in N^{1,p}(X)$ for which $f=0$ $C_p$-almost everywhere in $X\setminus W$,
see Definition~\ref{def:zero-boundary}.

\begin{lem}[{\cite[Lemma 3.10]{BBS}}]\label{lem310}
 Let $B\subset X$ be a ball and let
 $u\in \Np(B)$ be a $p$-superminimizer in $B$ with $0\leq u \leq 1$. 
 Then there exists a representative of $u$ that is lower semicontinuous at every point of $B$.
\end{lem}

See Definition~\ref{def:H-Omega-f} for the definition of $H_\Om f$ for functions $f\in N^{1,p}(\Om)$.

\begin{lem}\label{lem311}
 Let $\Om \subset X$ be a domain,  $B\subset X$ a ball with center in $\partial\Om$, and $W$ be a given 
 connected component of $B\cap \Om$. Suppose that $F: \clOmP\to \R$ is Lipschitz quasicontinuous 
 on $\Om\cup\bdySP\Om$ with  respect to the Mazurkiewicz metric and
 $F([E_k])=1$ whenever $[E_k]\in\bdySP\Om$ such that there is some $k_0\in\N$ with $E_k\subset W$
 for $k\ge k_0$. Suppose in addition that $F$
satisfies $0\leq F\leq 1$ and is such that $F|_{\Om}\in \Np(\Om)$. Define a function $\Psi: B \to \R$ as follows:
 \begin{equation}
\Psi(x)=
\begin{cases}
&  H_{\Om} F(x),\quad x\in B\cap W,\\
& 1,\quad x\in B\setminus W.
\end{cases}
\end{equation}
Then $\Psi\in N^{1,p}(B)$ and is a $p$-superminimizer in $B$.
\end{lem}

\begin{proof}
We first show that $\Psi\in N^{1,p}(B)$. Since $0\le\Psi\le 1$, it suffices to show that $\Psi$ has a $p$-weak
upper gradient in $L^p(B)$.  
As $F$ is $L$-Lipschitz on $\Om\cup\bdySP\Om$ with respect to the Mazurkiewicz metric $d_M$ for
some $L>0$, the constant function $g=L$ is an upper gradient of $F$ on $\Om\cup\bdySP\Om$
when considered with respect to the metric $d_M$. Then $g$ is an upper gradient, also with respect to
the original metric $d$, for $F$ on $\Om$. Furthermore, if $\gamma:[a,b]\to B\cap\overline{W}$ with $\gamma$ lying
entirely in $\Om$ or entirely in $B\cap\partial W$, then the pair $F,g$ satisfies the upper gradient gradient
inequality along $\gamma$. If $\gamma$ lies partially in $\Om$ and intersects $B\cap\partial W$, then by
splitting it into two parts if necessary we may assume that $\gamma(a)\in W$, and set 
$t_0:=\sup\{t\in[a,b]\, :\, \gamma([a,t])\subset W\}$. Note then that $\gamma(t_0)\in B\cap \partial W\subset \partial \Om$
but $\gamma([a,t_0))\subset W\subset\Om$. Therefore there is some singleton prime end
$[E_k]\in\bdP\Om$ such that $E_k\subset W$ for each $k\in\N$ and $I([E_k])=\{\gamma(t_0)\}$, and so
$\lim_{t\to t_0^-}F(\gamma(t))=1=\Psi(\gamma(t_0))$, and the pair $F,g$ satisfies the upper gradient inequality on 
$\gamma\vert_{[a,t]}$ for each $t<t_0$. It follows now that 
\[
|F(\gamma(a))-1|\le \int_{\gamma\vert_{[a,t_0]}}g\, ds.
\]
It follows that the extension of $F\vert_W$ by $1$ to $B\setminus W$ has $g\chi_W$ as an upper gradient in 
$B$. Note that the zero extension of
$F-H_\Om F$ from $\Om$ to $X$, denoted $h$, is in $N^{1,p}_0(\Om)$; let $g_h\in L^p(X)$ be an upper gradient 
of $h$ (in $X$). Then by the above argument,
the function $g_h+g\, \chi_W$ is an upper gradient of $\Psi$ in $B$, that is, $\Psi\in N^{1,p}(B)$.

 The rest of the proof follows the steps of the corresponding proof of~\cite[Lemma~3.11]{BBS}. 
 By Theorem~\ref{thm77} we get that $F$ is resolutive and $P_{\clOmP}F=H_{\Om}F$. The comparison 
 principle allows us to infer that $0\leq H_{\Om}F\leq 1$. Let $\phi\in \Npz(B\cap \overline{W})$ be nonnegative. 
 Our goal is to show that
\begin{equation}\label{eq1: lem311}
 \int_{B\cap \{\supp \phi\not =0\}} g_{\Psi}^p\leq  \int_{B\cap \{\supp \phi\not =0\}} g_{\Psi+\phi}^p.
\end{equation} 
  Since $\Psi\leq 1$ in $\overline{W}$, then we may assume that $\Psi+\phi\leq 1$, 
  as otherwise one can replace test function $\phi$ by $\min\{\phi, 1-\Psi\}$ and decrease 
  the right-hand side of~\eqref{eq1: lem311}. Hence, $\phi\equiv 0$ in $B\setminus W$ and thus 
  $\phi\in \Npz(B\cap \Om)$. As $H_{\Om}F$ is a $p$-minimizer in $B\cap \Om$ $\phi\in \Npz(B\cap\Om)\subset\Npz(\Om)$,
  it follows that
\begin{align*}
 \int_{B\cap \{\supp \phi\not =0\}} g_{\Psi}^p&=\int_{B\cap \Om\cap \{\supp \phi\not =0\}} g_{H_{\Om}F}^p \\
 & \leq \int_{B\cap \Om\cap \{\supp \phi\not =0\}} g_{H_{\Om}F+\phi}^p \\
 & = \int_{B \cap \{\supp \phi\not =0\}} g_{\Psi+\phi}^p.  
\end{align*}
 Thus $\Psi$ is a $p$-superminimizer in $B$.
\end{proof}

\begin{deff}\label{def:Res-Kellogg}
We say that a point $[E_k]\in\bdP\Om$ is  \emph{resolutively regular} if whenever $\phi\in C(\bdP\Om)$ is
resolutive, we have
\[
\lim_{\Om\ni y\to[E_k]}P_{\overline{\Om}^P}\pip=\pip([E_k]).
\]
We say that $[E_k]$ is resolutively irregular if it is not resolutively regular. Let $\Irr(\Om)$ denote the collection
of all resolutively irregular points in $\bdP\Om$.
\end{deff}

\begin{thm}[The Kellogg property for prime end boundary]\label{thm:Kellogg}
 Let $\Om\subset X$ be a bounded domain. Then 
$\CapPp(\Irr(\Om))=0$. 
\end{thm} 

Before the proof of the theorem, let us discuss some auxiliary definitions and results. For a set 
$A\subset \overline{\Om}$ we define ${\rm Pr}(A)$ as a subset of $\bdP \Om$ consisting of prime ends 
whose impressions belong to $A$:
  \[
    {\rm Pr}(A):=\{[E_m]\in \bdP \Om:  I[E_m]\subset A\}.
  \] 
  
\begin{proof}
 The proof follows the steps of the corresponding proof of Theorem 3.9 in \cite{BBS}. 
 However, the setting of prime end boundary requires several modifications of the original argument.
 
 The compactness of $\partial \Om$ allows us to find a finite covering of it by balls $B_{j,k}:=B(x_{j, k}, 2^{-j})$ 
 for $j=1,2,\ldots$ and $1\leq k \leq N_j$, with $x_{j,k}\in\partial\Om$. Note 
 that $2B(x_{j,k},  2^{-j})\cap\Om$ has at most countably many connectedness components which we 
 index with $l=1,2,\ldots$ and denote by $B_{j,k}^l$.
 
We introduce the functions $\phi_{j,k}^l: X\to \R$ by 
\[
 \phi_{j,k}^l(y)=\min\{1, 2^j \dist(y, X\setminus 2B(x_{j,k}, 2^{-j}))\}\chi_{\overline{B_{j,k}^l}}(y).
\]  
 It is easy to see that  $0\leq \phi_{j,k}^l \leq 1$ and that every $\phi_{j,k}^l$ is Lipschitz continuous on $\Om$, and hence
 is Lipschitz continuous on $\Om$ also with respect to the Mazurkiewicz metric $d_M$. 
 We extend  $\phi_{j,k}^l$ to functions on $\clOmP$ as follows:
 \begin{equation*} 
 \widehat{\phi_{j,k}^l}(y)=
 \begin{cases}
 \phi_{j,k}^l(y) &\text{ if } y\in \Om,\\
 \inf_{z\in I[y]} \phi_{j,k}^l(z) &\text{ if } y=[E_n]\in \bdP \Om \text{ with } E_n\subset B_{j,k}^l \hbox{ for large enough }n,  \\
 0 &\text{ otherwise}.
 \end{cases}
 \end{equation*}
 
Observe that these functions satisfy the following properties:
\begin{itemize}
\item[(i)]  $\widehat{\phi_{j,k}^l}$ is Lipschitz continuous with respect to $d_M$ on 
$\Om\cup \bdySP \Om$. Indeed, since we know that $\phi_{j,k}^l=\widehat{\phi_{j,k}^l}$ is 
Lipschitz on $\Om$ with respect to the metric $d_M$, it has a unique Lipschitz (with respect to
$d_M$) extension to $\Om\cup\bdySP\Om$, and as $\widehat{\phi_{j,k}^l}$ is continuous on
$\Om\cup\bdySP\Om$, it follows that it is Lipschitz on that set with respect to $d_M$.
\item[(ii)]  $\widehat{\phi_{j,k}^l}|_{\Om}=\phi_{j,k}^l|_\Om\in \Np(\Om)$. 
\item[(iii)] functions $\widehat{\phi_{j,k}^l}$ are $\CapPp$-quasicontinuous on $\clOmP$. 
This follows immediately from part (i) above and Theorem~\ref{prop:exceptional}. 
\item[(iv)] If $[E_m]\in\bdySP\Om$ with $E_m\subset B_{j,k}^l$ for large $m\in\N$, 
then $\widehat{\phi_{j,k}^l}([E_m])=\phi_{j,k}^l(x)$ where $\{x\}=I([E_m])$.
\end{itemize}
 
  The above properties allow us to employ Theorem~\ref{thm77} and conclude that 
$\widehat{\phi_{j,k}^l}$ is
$\CapPp$-quasicontinuous  and is resolutive with
 \[
  P_{\clOmP} \widehat{\phi_{j,k}^l}= H_{\Om} \widehat{\phi_{j,k}^l}=H_{\Om} \phi_{j,k}^l.
 \]
Set 
  \begin{equation*} 
   I_{j,k,l}=\{[E_m]\in \bdySP \Om \cap {\rm Pr}(B_{j,k} \cap \overline{B_{j,k}^l})\, : \, 
      \liminf_{\Om\ni y\to [E_m]} H_{\Om} \phi_{j,k}^l(y)< \widehat{\phi_{j,k}^l}([E_m])=1\}.
  \end{equation*}
 
 Note that $I_{j,k,l}\subset\Irr(\Om)$.  
 Moreover, by Lemma~\ref{lem311} 
 applied to $B=B_{j,k}$, component $W=B_{j,k}^l$ and $F=\widehat{\phi_{j,k}^l}$ we get $\Psi$, a 
 $p$-superminimizer on $B_{j,k}^l$ such that $\Psi=H_{\Om} \phi_{j,k}^l$ on $B_{j,k}\cap B_{j,k}^l$ and 
 $\Psi\equiv 1$ on $B_{j,k}\setminus B_{j,k}^l$. 
 Therefore, the assertion of Lemma~\ref{lem310} together with~\cite[Lemma~6.11]{ES} allows us to infer that 
 \begin{equation}\label{eq:cap-cap-irr}
0\le \CapPp(I_{j,k,l})\leq \Cp(P^{-1}(I_{j,k,l}))=0, 
 \end{equation}
 where
 \[
 P^{-1}(I_{j,k,l})=\{x\in\partial\Om\, :\, \{x\}=I([E_m])\text{ for some }[E_m]\in I_{j,k,l}\}.
 \]  
 Let $f\in C(\bdP \Om)$ and let $[E_m]\in \bdySP \Om$ for which
\[
\lim_{\Omega\ni y\overset{\clOmP} \to [E_m]} P_{\overline{\Om}^P}f(y) \not =f([E_m]).
\]
We should interpret the above statement to mean that either the limit on the left-hand side does not exist,
or, if it exists, fails to equal the value on the right-hand side.
 Let $x_0=I([E_m])$.  
 By adding a constant to $f$, and by scaling $f$ if necessary, we may assume that $f\ge 0$ on $\bdP\Om$,
 $f([E_m])>1$ and that 
 \[
 \liminf_{\Omega\ni y\overset{\clOmP} \to [E_m]} P_{\overline{\Om}^P}f(y)<1. 
 \]
 As $f$ is continuous
 on $\bdP\Om$, there is some open set $U\subset\overline{\Om}^P$ with $[E_m]\in U$ such that
 $f>1$ on $\bdP\Om\cap U$. For each $j\in\N$ there is some $k_j$ such that $x_0\in B_{j,k_j}$.
 
 As $[E_m]\in\bdySP\Om$, we have $\lim_{m\to\infty}\text{diam}(E_m)=0$. Thus $E_m\subset B_{j,k_j}$
 for sufficiently large $m\in\N$, and as $E_m$ is connected and $x_0\in\partial E_m\cap\partial\Om$,
 it follows that
 there is some $l$ with $x_0\in B_{j,k_j}\cap\partial B_{j,k_j}^l$ and
 for sufficiently large $m\in\N$ we have $E_m\subset B_{j,k_j}^l\subset U$. It follows that 
 $\widehat{\phi_{j,k_j}^l}([E_m])=1$ and $\widehat{\phi_{j,k_j}^l}\le f$ on $\bdP\Om$. It follows that
 \[
 H_\Om\widehat{\phi_{j,k_j}^l}= P_{\overline{\Om}^P}\widehat{\phi_{j,k_j}^l}\le P_{\overline{\Om}^P}f
 \]
 on $\Om$. It follows that
 \[
 \liminf_{\Omega\ni y\overset{\clOmP} \to [E_m]} P_{\overline{\Om}^P}\widehat{\phi_{j,k_j}^l}(y)\le 
 \liminf_{\Omega\ni y\overset{\clOmP} \to [E_m]} P_{\overline{\Om}^P}f(y)<1,
 \]
 That is, $[E_m]\in I_{j,k_j,l}$. Thus the collection of all $[E_m]\in\bdySP\Om$ for which
 \[
 \lim_{\Omega\ni y\overset{\clOmP} \to [E_m]} P_{\overline{\Om}^P}f(y) \not =f([E_m])
 \]
 is a subset of $\bigcup_{j\in\N}\bigcup_k\bigcup_l I_{j,k,l}$, and this holds true for each 
 resolutive $f\in C(\bdP\Om)$. 
 Hence $\Irr(\Om)\cap\bdySP\Om\subset \bigcup_{j\in\N}\bigcup_k\bigcup_l I_{j,k,l}$,
 and so by~\eqref{eq:cap-cap-irr},
 Theorem~\ref{prop:exceptional}, and by the countable subadditivity of $\CapPp$ the desired conclusion follows.
\end{proof}


\begin{thebibliography}{99}

\bibitem{abbs} {\sc T. Adamowicz, A. Bj\"orn, J. Bj\"orn \AND N. Shanmugalingam}, 
Prime ends for domains in metric spaces, 
\emph{Adv. Math.} 238 (2013), 459--505.

\bibitem{aw} {\sc T. Adamowicz \AND B. Warhurst}, Prime ends in the Heisenberg group $\mathbb{H}_1$ and the boundary behavior of quasiconformal mappings, \emph{Ann. Acad. Sci. Fenn. Math.} 43(2) (2018), 1--38.


\bibitem{AT} {\sc L. Ambrosio \AND P. Tilli},
\emph{Topics on analysis in metric spaces},
Oxford Lecture Series in Mathematics and its Applications 25, Oxford University Press,
Oxford, 2004.

\bibitem{BjComb}{\sc A. Bj\"orn}, The Dirichlet problem for $p$-harmonic functions on the topologist's comb,
\emph{Math. Z.}, 279 (2015), no. 1-2, 389--405.

\bibitem{Bj17}{\sc A. Bj\"orn}, The Kellogg property and boundary regularity for $p$-harmonic functions 
with respect to the Mazurkiewicz boundary and other compactifications, 
to appear in \emph{Complex Var. Elliptic Equ.} {\tt doi:10.1080/17476933.2017.1410799}

\bibitem{BBbook} \book{A. Bj\"orn \AND J. Bj\"orn}
        {\it Nonlinear Potential Theory on Metric Spaces}
       {EMS Tracts in Mathematics {17},
        European Math. Soc., Zurich}

\bibitem{bbs1} {\sc A. Bj\"orn, J. Bj\"orn \AND N. Shanmugalingam}, The Dirichlet problem for $p$-harmonic 
functions with respect to the Mazurkiewicz boundary, and new capacities, 
\emph{J. Differential Equations}, 259(7) (2015), 3078--3114.

\bibitem{BBS} \art{A. Bj\"orn, J. Bj\"orn \AND N. Shan\-mu\-ga\-lin\-gam}
       {The Dirichlet problem for \p-harmonic functions on metric spaces}
        {J. Reine Angew. Math.} {556} {2003} {173--203}

\bibitem{BBS2} \art{A. Bj\"orn, J. Bj\"orn \AND N. Shan\-mu\-ga\-lin\-gam}
        {The Perron method for \p-harmonic functions}
        {J. Differential Equations} {195} {2003} {398--429}

\bibitem{Cheeg} \art{J. Cheeger}
        {Differentiability of Lipschitz functions on metric spaces}
        {Geom. Funct. Anal.} {9} {1999} {428--517}

\bibitem{cl} \book{E. Collingwood, \AND E. Lohwater}
 {The Theory of Cluster Sets}
 {Cambridge Tracts in Math. and Math. Phys. 56, Cambridge Univ. Press, Cambridge, 1966}

\bibitem{Est} \book{D. Estep}
  {Prime end boundaries of domains in metric spaces, and the Dirichlet problem}
  {Thesis (Ph.D.)--University of Cincinnati, (2015), 95 pp}

\bibitem{ES} \art{D. Estep \AND N. Shanmugalingam}
    {Geometry of prime end boundary and the Dirichlet problem for bounded domains in metric measure spaces}
    {Potential Analysis}{42}{2015}{335--363}

\bibitem{HaKo} \book{P. Haj\l asz, 
	\AND P. Koskela}
	{Sobolev met Poincar\'e}
	{{Mem. Amer. Math. Soc.} {145}, 688 (2000)}

\bibitem{Hed} \art{L.-I. Hedberg}
    {Non-linear potentials and approximation in the mean by analytic functions}
    {Math. Z.}{129}{1972}{299--319}

\bibitem{heinonen} \book{J. Heinonen}
        {Lectures on Analysis on Metric Spaces}
        {Springer-Verlag, New York, 2001}

\bibitem{he07} \art{J. Heinonen}{Nonsmooth calculus}
{Bull. Amer. Math. Soc. (N.S.)}{44(2)}{2007}{163--232}

\bibitem{HeKiMa} \book{J. Heinonen, T. Kilpel\"ainen, \AND O. Martio}
        {Nonlinear Potential Theory of Degenerate Elliptic Equations}
        {2nd ed., Dover, Mineola, NY, 2006}

\bibitem{HeKo98} \art{J. Heinonen, \AND P. Koskela}
	{Quasiconformal maps in metric spaces with controlled geometry}
	{Acta Math.} {181} {1998} {1--61}

\bibitem{HKST} \book{J. Heinonen, P. Koskela, N. Shanmugalingam, \AND J. Tyson} 
{Sobolev spaces on metric measure spaces. An approach based on upper gradients} 
{New Mathematical Monographs, 27. Cambridge University Press, Cambridge, 2015}

\bibitem{KiMa02} \art{J. Kinnunen \AND O. Martio}
        {Nonlinear potential theory on metric spaces}
        {Illinois Math. J.} {46} {2002} {857--883}

\bibitem{KiSh01} \art{J. Kinnunen, \AND N. Shan\-mu\-ga\-lin\-gam}
        {Regularity of quasi-minimizers on metric spaces}
        {Manuscripta Math.} {105} {2001} {401--423}

\bibitem{mcs} \art{E. J. McShane}
{Extension of range of functions}
{Bull. Amer. Math. Soc.}{40(12)}{1934}{837--842}

\bibitem{na} \art{R. N\"akki}
     {Prime ends and quasiconformal mappings}
     {J. Anal. Math.} {35} {1979} {13--40}

\bibitem{Sh-rev} \art{N. Shan\-mu\-ga\-lin\-gam}
        {Newtonian spaces\textup{:} An extension of Sobolev spaces
        to metric measure spaces}
        {Rev. Mat. Iberoam.}{16}{2000}{243--279}

\end{thebibliography}
\end{document}